\documentclass[a4paper]{amsart}

\usepackage[foot]{amsaddr}
\usepackage{amsthm,amssymb}

\usepackage{tikz}
\usetikzlibrary{shapes,arrows,positioning}

\usepackage[utf8]{inputenc}
\usepackage{amsmath,amssymb,amsthm,commath,esint,tikz-cd,tikz, mathtools, physics}
\usepackage[margin=1in]{geometry}
\usepackage[mathscr]{euscript}
\usepackage{amsfonts}
\usepackage{subcaption}
\usepackage{hyperref}
\usepackage{graphicx}
\usepackage{float}
\usepackage{fancyhdr}
\usepackage{bm}

\title[Neural network approximations of parametric hyperbolic conservation laws]{Error analysis for deep neural network approximations of parametric hyperbolic conservation laws}

\author{T.~De Ryck}
\email{tim.deryck@sam.math.ethz.ch}

\author{S.~Mishra}
\address[T. De Ryck and S. Mishra]{Seminar for Applied Mathematics, ETH Z\"urich, R\"amistrasse 101, 8092 Z\"urich, Switzerland}
\email{siddhartha.mishra@sam.math.ethz.ch}

\newcommand{\s}{s} 
\renewcommand{\S}{\mathcal{S}} 
\newcommand{\Et}{\mathcal{E}_T} 
\newcommand{\Eg}{\mathcal{E}_G} 
\newcommand{\Aet}{\overline{\mathcal{E}}_T} 
\newcommand{\Aeg}{\overline{\mathcal{E}}_G} 
\newcommand{\ctv}{C_{TV}}

\newtheorem{theorem}{Theorem}[section]

\newtheorem{example}[theorem]{Example}
\newtheorem{remark}[theorem]{Remark}
\newtheorem{definition}[theorem]{Definition}
\newtheorem{lemma}[theorem]{Lemma}

\newtheorem{corollary}[theorem]{Corollary}

\begin{document}

\begin{abstract}
We derive rigorous bounds on the error resulting from the approximation of the solution of parametric hyperbolic scalar conservation laws with ReLU neural networks. We show that the approximation error can be made as small as desired with ReLU neural networks that overcome the curse of dimensionality. In addition, we provide an explicit upper bound on the generalization error in terms of the training error, number of training samples and the neural network size. The theoretical results are illustrated by numerical experiments. 
\end{abstract}

\maketitle

\section{Introduction}
A large number of interesting phenomena in science and engineering are modelled by a class of nonlinear PDEs called \emph{hyperbolic systems of conservation laws} \cite{holden2015front}. Examples include the compressible Euler equations of aerodynamics, the shallow-water equations of oceanography, the magnetohydrodynamics (MHD) equations of plasma physics and the equations of nonlinear elasticity. It is well-known that solutions of hyperbolic conservation laws develop discontinuities such as shock waves and contact discontinuities in finite time, even when the initial data are smooth. Thus, solutions are sought in the weak (distributional) sense. However, these weak solutions are not necessarily unique and additional admissible criteria or \emph{entropy conditions} are imposed in order to recover uniqueness. 

A variety of numerical methods have been proposed to approximate the entropy solutions of conservation laws efficiently. These include high-resolution finite volume methods, based on TVD, ENO and WENO reconstructions, discontinuous Galerkin and spectral viscosity methods \cite{HEST1}. Despite their well-documented successes, these standard numerical methods are characterized by their high computational cost, particularly in two and three space dimensions \cite{LMR1} and references therein. This cost is particularly evident when the problem at hand requires \emph{multiple} calls to the PDE solver. Such \emph{many-query} problems can appear in a variety of contexts such as Uncertainty quantification (UQ) \cite{UQbook}, PDE constrained optimization \cite{Popt}, deterministic and Bayesian Inverse problems \cite{BINV}. Using conventional numerical methods for these problems is unfeasible and one needs to develop \emph{fast surrogates} which can accurately approximate the PDE solution at a very small fraction of the computational cost of a conventional numerical method for hyperbolic PDEs, particularly in several space dimensions. 

Many different approaches have been tried in the context of designing fast and efficient surrogates for conservation laws. One class of methods can be classified as \emph{reduced order models} or \emph{reduced basis methods}, see \cite{ROMbook} and references therein for further details of this approach. In these methods, the surrogate is formed from a linear combination of \emph{principal modes}, which can be determined from different orthogonal decompositions of the data based \emph{snapshot matrix}. However, these methods are known to be inefficient in transport-dominated problems such as hyperbolic conservation laws due to the \emph{slow decay} of the underlying eigenvalues, characterizing the principal modes, \cite{Ab1,Ab2} and references therein. 

Another class of surrogate models for PDEs is based on \emph{deep neural networks} (DNNs). The literature on the use of DNNs as surrogates for PDEs has grown exponentially in the last couple of years and a very selected list includes \cite{SZ1,Kuty} for linear elliptic PDEs, \cite{HEJ1} for linear and semi-linear parabolic PDEs, \cite{KAR1,KAR2} for physics informed neural networks for PDEs and \cite{deeponets,FNO} for learning the underlying solution operator for PDEs. 

DNNs have recently been explored in the context of designing surrogates for hyperbolic PDEs. References include \cite{LMR1} where the authors use DNNs to approximate observables, corresponding to systems of conservation laws, see also \cite{LMM1}, and utilize this surrogate for efficient UQ. Moreover, in \cite{LMPR1}, DNNs have been investigated as backbones in optimization algorithms, constrained by hyperbolic systems of conservation laws. 

The focus of most of the literature for DNN surrogates for hyperbolic conservation laws has been the design and empirical demonstration of these algorithms via numerical experiments. Consequently, there is a paucity of theoretical results in the form of error estimates for deep neural networks as surrogates for hyperbolic PDEs. Providing such rigorous estimates constitutes the main aim of this paper. 

We seek to provide rigorous bounds on the error incurred by DNNs when approximating possibly discontinuous solutions of hyperbolic conservation laws. To this end, we will focus on the prototypical case of \emph{scalar conservation laws}, where and in contrast to systems of conservation laws, a rigorous theory of well-posedness of entropy solutions as well convergence of numerical approximations is available. Moreover, we model \emph{many-query} problems by considering scalar conservation laws with \emph{parametric} initial conditions as well as \emph{parametric} flux functions. Such problems arise in the context of UQ as well as design or optimization where the parameters represent random variables or design variables, respectively. A classic example of such parameterizations, particularly in the context of UQ \cite{mishra2016numerical}, is provided by the so-called \emph{truncated Karhunen-Loève expansions}.  

With this setting, our main contributions in this paper are,

\begin{itemize}
    \item We prove that neural networks can approximate solutions to scalar conservation laws with parametric initial condition (Section \ref{sec:in-data}) and/or parametric flux function (Section \ref{sec:flux}) without incurring the curse of dimensionality with respect to the parameter space. We provide explicit error bounds in terms of the neural network size and prove that the network weights and biases remain bounded. 
    \item An explicit, computable upper bound on the so-called \emph{generalization error} is also provided. It shows that the curse of dimensionality is overcome in the number of training samples as well. 
    \item We provide numerical experiments to validate our theoretical observations. In all examples, the training and generalization error only grow  polynomially, at worst, in the input dimension and the generalization gap converges as the data set size grows, as predicted by the theory. 
\end{itemize}
Thus, we provide rigorous analysis of the different sources of error in the approximation of parametric scalar conservation laws by DNNs and show that these errors can only grow polynomially in the parametric dimension. 

The rest of this paper is organized as follows. In Section \ref{sec:prel} we collect preliminary material such as definitions and basic neural network approximation results. The main results on the approximation error for parametric hyperbolic conservation laws by ReLU neural networks can be found in Section \ref{sec:approximation-error}. A bound on the generalization error is prove in Section \ref{sec:gen} and the numerical experiments are presented in Section \ref{sec:numex}. 

\section{Preliminaries on Neural Networks}\label{sec:prel}
Our fundamental aim in this paper is to study approximation of solutions of conservation laws by neural networks. To this end, we start by collecting some preliminary material and results on DNNs in this section.

\subsection{Definitions}

In this paper, we will consider function approximation using feedforward artificial neural networks where only connections between neighbouring layers are allowed. For simplicity, we will just refer to them as neural networks. In the following, we formally introduce our definition of a neural network and the related terminology. 

\begin{definition}\label{def:relu-act}
We denote by $\sigma:\mathbb{R}\to\mathbb{R}$ the function that satisfies for all $x\in\mathbb{R}$ that $\sigma(x) = \max\{x,0\}$ and we call $\sigma$ the ReLU activation function or rectifier function. For $n\in\mathbb{N}$ and $x\in\mathbb{R}^n$, we define $\sigma(x) := (\max\{x_1,0\},\ldots,\max\{x_n,0\})$. 
\end{definition}

\begin{definition}\label{def:relu-nn}
Let $L\in\mathbb{N}$ and $l_0,\ldots, l_L\in\mathbb{N}$. Let $\sigma$ denote the ReLU activation function and define the parameter space $\Theta$ as
\begin{equation}
    \Theta= \bigcup_{L\in\mathbb{N}}\bigcup_{l_0,\ldots,l_L\in\mathbb{N}}\bigtimes_{k=1}^L \left(\mathbb{R}^{l_k\times l_{k-1}}\times\mathbb{R}^{l_k}\right). 
\end{equation}
For $\theta\in\Theta$, we define $(W_k,b_k):=\theta_k$ and $\mathcal{A}_k:\mathbb{R}^{l_{k-1}}\to\mathbb{R}^{l_{k}}:x\mapsto W_k x+b_k$ for $1\leq k\leq L$ and we denote by $\Psi_\theta:\mathbb{R}^{l_0}\to\mathbb{R}^{l_L}$ the function that satisfies for all $x\in\mathbb{R}^{l_0}$ that
\begin{equation}
    \Psi_\theta(x) = \begin{cases}\mathcal{A}_1(x) & L=1\\ (\mathcal{A}_L \circ \sigma \circ \mathcal{A}_{L-1} \circ \sigma \circ\cdots \circ \sigma \circ \mathcal{A}_1)(x) & L>1. \end{cases}
\end{equation}
We refer to $\Psi_\theta$ as (the realization of) the ReLU neural network associated to the parameter $\theta$ with $L$ layers with widths $(l_1, \ldots, l_L)$, of which the first $L-1$ layers are called hidden layers. For $1\leq k\leq L$, we say that layer $k$ has width $l_k$ and we refer to $W_k$ and $b_k$ as the weights and biases corresponding to layer $k$. If $L\geq 3$, we say that $\Psi_\theta$ is a deep ReLU neural network. 
\end{definition}

We will be interested in proving the existence of neural networks of a certain size that approximate the solution of a PDE to a specified accuracy. 
Following \cite{deep-it-2019}, we will quantify the size of neural networks based on:
\begin{itemize}
    \item the connectivity $\mathcal{M}(\Phi_\theta) := \norm{\theta}_0$, i.e. the total sum of the number of non-zero entries in all weights and biases, 
    \item the depth $\mathcal{L}(\Phi_\theta) := L$, 
    \item the maximum width $\mathcal{W}(\Phi_\theta) := \max_{0\leq k \leq L}l_k$, 
    \item the weight magnitude $\mathcal{B}(\Phi_\theta) :=\norm{\theta}_{\infty}$. 
\end{itemize}
Note that this notation is in general not well-defined since the realization map $\theta\mapsto \Psi_\theta$ need not be injective. However, this will not pose a problem since in our proofs we will explicitly construct the parameter vector $\theta$ rather than proving the existence of a neural network $\Psi_\theta$.

Next, we introduce the concept of a clipped neural network, following e.g. \cite{berner2020analysis,jentzen2020overall}.
\begin{definition}\label{def:clipped-dnn}
Let $a,b,\alpha,\beta\in\mathbb{R}$ with $a<b$ and $\alpha<\beta$, $L\in\mathbb{N}$ and $l_0,\ldots, l_L\in\mathbb{N}$, $\theta\in\Theta$ and let $\Phi_\theta$ be the associated ReLU neural network realization (cf. Definition \ref{def:relu-nn}). We call $\Phi_\theta$ a $(\alpha,\beta)$-clipped ReLU neural network if it holds that $\Phi_\theta(x)\in[\alpha,\beta]^{l_L}$ for all $x\in[a,b]^{l_0}$. 
\end{definition}
Although the use of clipped neural networks will only become clear later on, one should note that any ReLU neural network $\Phi_\theta$ can be transformed into a $(\alpha,\beta)$-clipped ReLU neural network. To see this, consider the following clipping function
\begin{equation}
    \mathscr{C}: \mathbb{R}\to\mathbb{R}: x\mapsto \max\{\alpha,\min\{\beta,x\}\}. 
\end{equation}
Clearly, $\mathscr{C}(x)\in[\alpha,\beta]$ for all $x\in\mathbb{R}$ and since $\mathscr{C}$ can also be expressed in terms of the rectifier function, it can be realized as a ReLU neural network. Given that the composition of two ReLU neural networks is again a ReLU neural network, it holds that $\mathscr{C}\circ\Phi_\theta$ is a $(\alpha,\beta)$-clipped ReLU neural network for any $\theta\in\Theta$. 

Per definition, ReLU neural networks are mappings of which both the input and the output is a vector. In our setting, however, we are interested in expressing or approximating the \textit{parameter-to-solution} map $Y\to L^{\infty}(D):y\mapsto u(T,\cdot,y)$ using ReLU neural networks. In order to do this, we must identify the output vector of a ReLU neural network with a function in $L^{\infty}(D)$. This leads to the following definition.

\begin{definition}\label{def:sol-associated-dnn}
Let $J,d\in\mathbb{N}$, let $a,b\in\mathbb{R}$ with $a<b$, $\theta\in\Theta$ and let $\Phi_\theta$ be the ReLU neural network associated to $\theta$ with input dimension $d$ and output dimension $J$. In addition, let $D = [a,b]$ and define $x_{j+1/2} = a+j(b-a)/J$ and $C_j := [x_{j-1/2},x_{j+1/2})$ for $0 \leq j \leq J$. We define $u_\theta:D\times Y\to \mathbb{R}$ by setting
\begin{equation}
    u_\theta(x,y) = (\Phi_\theta(y))_j \quad \text{for } x\in C_j, y\in Y \text{ and } 1 \leq j \leq J.
\end{equation}
We call $u_\theta$ the solution function associated to the ReLU neural network $\Phi_\theta$. 
\end{definition}

\subsection{Function approximation with ReLU neural networks}

One of the crucial elements in the explanation of the popularity of ReLU neural networks is their expressivity. Under very mild assumptions on the network architecture, ReLU neural networks can approximate many function classes to arbitrarily small accuracy. The classical universal approximation theorem \cite{Cybenko1989,hornik1991,LESHNO1993} states for example that the class of ReLU neural networks with only one hidden layer is dense in the continuous functions on a compact interval. More recent work tries to provide insight in how the accuracy of the approximation depends on the regularity of the function class and the width and depth of the network. 

In the following sections, we will require such accuracy results, but not with respect to the often used $L^\infty$-norm, but to the Lipschitz seminorm. We define the Lipschitz seminorm for a function $f:[a,b]\to\mathbb{R}$ as
\begin{equation}
    \norm{f}_{\mathrm{Lip}([a,b];\mathbb{R})} = \sup_{x,y\in[a,b], x\neq y} \frac{\abs{f(x)-f(y)}}{x-y}.
\end{equation}

Next, we prove a result on the approximation of $C^2$ functions in Lipschitz seminorm by ReLU neural networks. The proof can be found in Appendix \ref{app:flux-approximation}.

\begin{lemma}\label{lem:flux-approximation}
Let $a,b\in\mathbb{R}$ with $a<b$ and let $f\in C^2([a,b];\mathbb{R})$. For any $J\in\mathbb{N}$, there exists a realization of a ReLU neural network $\Phi^{J}$ such that
\begin{equation}
    \norm{f-\Phi^J}_{\mathrm{Lip}([a,b];\mathbb{R})}\leq \frac{b-a}{J}\norm{f''}_\infty.
\end{equation}
In addition, $\Phi^J$ satisfies that $\mathcal{M}(\Phi^J) = O(J)$, $\mathcal{L}(\Phi^J) = 2$, $\mathcal{W}(\Phi^J) = J+3$ and $\mathcal{B}(\Phi^J) \leq \max \{1,\abs{a}+\abs{f(a)},\abs{b}+\abs{f(b)},2 \norm{f'}_\infty\}$. 
\end{lemma}

Next, we adapt Yarotsky's result on the approximation of the multiplication operator by ReLU neural networks \cite{yarotsky2017error} from supremum norm to Lipschitz seminorm. The proof can be found in Appendix \ref{app:mult-lipschitz}.

\begin{lemma}\label{lem:mult-lipschitz}
Let $M,N>0$ and denote by $\times:\mathbb{R}^2\to\mathbb{R}:(x,y)\mapsto xy$ the multiplication operator. For any $m\in\mathbb{N}$, there exists a realization of a ReLU neural network $\widehat{\times}_m:[-M,M]\times[-N,N]\to \mathbb{R}$ such that $\widehat{\times}_m$ satisfies for all $y\in[-N,N]$ the error bound
\begin{equation}
    \norm{\times(\cdot,y)-\widehat{\times}_m(\cdot,y)}_{\mathrm{Lip}([-M,M];\mathbb{R})}\leq \frac{M+N}{2^{m+1}}. 
\end{equation}
Furthermore, it holds that $\mathcal{M}(\widehat{\times}_m) = O(m)$, $\mathcal{L}(\widehat{\times}_m) = m+1$, $\mathcal{W}(\widehat{\times}_m) = 8$ and $\mathcal{B}(\widehat{\times}_m) = O((M+N)^2)$. 
\end{lemma}

We note that, in terms of complexity, sharper results than Lemma \ref{lem:flux-approximation} are available in the literature, especially for functions of higher regularity. As this exact complexity will not be very important later on, we refrain from discussing these results. In addition, Lemma \ref{lem:flux-approximation} has two advantages. Suppose one is interested in approximating some function to accuracy $\epsilon>0$ using neural network $\Phi^\epsilon$. First, the exact network size and structure of the approximation is stated, whereas in the literature, the exact network size is often dependent on unknown constants that might be very large. Therefore, networks that are asymptotically smaller in the limit $\epsilon \to 0$, when compared to those considered in Lemma \ref{lem:flux-approximation}, can be actually larger for realistic values of $\epsilon$. Second, a sharper complexity comes often at the cost of $\mathcal{B}(\Phi^\epsilon) \to \infty$ when $\epsilon\to 0$, e.g. \cite{deryck2021approximation}. In practice however, the weights and biases of the network are often regularized to avoid that $\mathcal{B}(\Phi^\epsilon) \to \infty$. Lemma \ref{lem:flux-approximation} and Lemma \ref{lem:mult-lipschitz} are therefore highly relevant as the magnitudes of the weights and biases of the constructed networks are bounded uniformly in $\epsilon$.

\section{ReLU DNN approximation of the solution of parametric scalar hyperbolic conservation laws}\label{sec:approximation-error}

In this section we prove our main results on the approximation of the solution of scalar conservation laws with parametric initial data and flux by ReLU neural networks. We show in particular that the constructed network overcomes the curse of dimensionality as its size depends only polynomially on the input dimension of the network. When the initial data and the flux follow a Karhunen-Loève expansion, the size can be seen to only depend \emph{linearly} on the input dimension of the network. We first consider scalar conservation laws with parametric initial data and a fixed flux and generalize to parametric flux functions thereafter.

\subsection{Parametric initial data, fixed flux}\label{sec:in-data}

We consider the one-dimensional scalar conservation law with parametric initial condition, 

\begin{equation}\label{eqn:original-scl}
\begin{cases}
\partial_t u(t,x,y) + \partial_x f(u(t,x,y)) = 0,\\
u(0,x,y) = u_0(x,y),
\end{cases}
\end{equation}
for all $x\in D=[a,b]$, $y\in Y=[0,1]^d$ and $t\in[0,T]$. We make the following assumptions on the flux and the initial data: 
\begin{itemize}
    \item[\textit{(A1)}] We assume that $f\in C^2(\mathbb{R};\mathbb{R})$. 
    \item[\textit{(A2)}] We assume that that $u_0\in L^{\infty}(D\times Y)$ and that  $\sup_{y\in Y} \mathrm{TV}(u_0(\cdot,y)) <\infty$. 
    \item[\textit{(A3)}] We assume that the map $u_0(x,\cdot)$ can be accurately approximated by a ReLU neural network for every $x\in D$. More precisely, we assume that there exists a constant $C_{\mathcal{B}}(u_0)>0$ such that for every $x\in D$ and $\epsilon>0$ there exists a ReLU neural network $\widehat{u}_0^\epsilon (x,\cdot)$ with $\mathcal{B}(\widehat{u}_0^\epsilon (x,\cdot))\leq C_{\mathcal{B}}$ that satisfies the error bound
    \begin{equation*}
        \norm{u_0(x,\cdot)-\widehat{u}_0^\epsilon (x,\cdot)}_{L^1(Y;\mathbb{R})}\leq \epsilon. 
    \end{equation*}
    In addition, the connectivity, depth and maximal width of the networks should only grow polynomially in $d$, i.e. for $\mathcal{Q}\in\{\mathcal{M},\mathcal{L}, \mathcal{W}\}$ there should exist $C_{\mathcal{Q}}(u_0),\sigma_{\mathcal{Q}}(u_0),\eta_{\mathcal{Q}}(u_0)>0$ such that it holds for all $x\in D$ that
    \begin{equation*}
        \mathcal{Q}(\widehat{u}_0^\epsilon (x,\cdot)) \leq C_{\mathcal{Q}} d^{\sigma_{\mathcal{Q}}}\epsilon^{-\eta_{\mathcal{Q}}}.
    \end{equation*}
    Finally, we require that
    \begin{equation*}
        \begin{split}
        C_0 &= \sup_{\epsilon>0}\max\{\norm{u_0}_{L^{\infty}(D\times Y)},\norm{\widehat{u}_0^\epsilon }_{L^{\infty}(D\times Y)}\}<\infty, \\
        \ctv &= \sup_{y\in Y, \epsilon>0}\max\{\mathrm{TV}(u_0(\cdot,y)), \mathrm{TV}(\widehat{u}_0^\epsilon (\cdot,y))\}<\infty.
    \end{split}
    \end{equation*}
\end{itemize}
The third assumption guarantees that $u_0$ can be approximated by ReLU neural networks without the curse of dimensionality. Examples of functions for which (A3) is satisfied include functions with a compositional structure or functions for which the relevant input data lies on a lower-dimensional manifold. Another class of functions that is especially relevant for conservation laws is highlighted in the following example. 

\begin{example}\label{ex:kle}
Assumption (A3) is satisfied by (random) initial data that admits a Karhunen-Loève expansion \cite{mishra2016numerical}, as one can then truncate this expansion up to $d$ terms to retrieve initial data of the form,
\begin{equation}
    u_0(x,y) = \overline{u}(x)+\sum_{i=1}^d \sqrt{\lambda_i}y_i\varphi_i(x),
\end{equation}
where $\overline{u}\in L^{\infty}(D)\cap BV(D)$, 
$y_i\in [0,1]$ and where $\varphi_i$ are some basis functions for all $i$. Note that the constant $d$ now reflects the dimensionality of the initial data. For every $x\in D$, $u_0(x,\cdot)$ can be trivially realized by a ReLU neural network with 0 hidden layers and $d+1$ non-zero weights and biases. In this case, it holds for all $x\in D$ that $\mathcal{M}(\widehat{u}_0^\epsilon (x,\cdot))=d+1$, $\mathcal{L}(\widehat{u}_0^\epsilon (x,\cdot))=1$ and $\mathcal{W}(\widehat{u}_0^\epsilon (x,\cdot))=d$.
\end{example}
Under the aforementioned three assumptions, we prove that the solution of \eqref{eqn:original-scl} at time $t=T$ can be approximated using a ReLU neural network. Moreover, the size of this network and its number of parameters only grow polynomially in the dimension $d$ of the parameter space of the initial data.

\begin{theorem}\label{thm:scl-dnn-1d}
Let $T>0$, $d\in\mathbb{N}$, $a,b\in\mathbb{R}$ with $a<b$ and $D=[a,b]$. 
For $y\in Y=[0,1]^d$, denote by $u(T,\cdot,y)$ the solution at time $t=T$ of \eqref{eqn:original-scl} with a flux $f:[-C_0,C_0]\to\mathbb{R}$ that satisfies assumption (A1) and initial condition $u_0(\cdot, y)$ that satisfies assumptions (A2) and (A3). Then for every $N\in\mathbb{N}$ there exists a ReLU neural network such that the associated solution map $\widehat{\mathcal{U}}^N:Y\to\mathbb{R}$ (cf. Definition \ref{def:sol-associated-dnn}) satisfies the error bound
\begin{equation}\label{eqn:expressivity-dnn-scl}
    \sup_{y\in Y}\norm{u(T,\cdot,y)-\widehat{\mathcal{U}}^N(\cdot,y)}_{L^1(D)} \leq  \frac{2\ctv T\left(C_0\norm{f''}_\infty+18\left(1+\norm{f'}_\infty\right)^{2}\right)+1}{\sqrt{N}}.
\end{equation}
In addition, it holds that that for $N\in\mathbb{N}$,
\begin{align}
    \begin{split}
        \mathcal{M}(\widehat{\mathcal{U}}^N) &= O \left( d^{\sigma_{\mathcal{M}}}N^{1+\eta_{\mathcal{M}}/2}+N^{5/2}\right), \\
        \mathcal{L}(\widehat{\mathcal{U}}^N) &\leq C_\mathcal{L} d^{\sigma_{\mathcal{L}}}N^{\eta_{\mathcal{L}}/2}+N, \\
        \mathcal{W}(\widehat{\mathcal{U}}^N), &\leq  \frac{2(b-a)}{T\norm{f'}_\infty}\max\left\{1+C_\mathcal{W} d^{\sigma_{\mathcal{W}}}N^{1+\eta_{\mathcal{W}}/2},N^{3/2}\right\}, \\ 
         \mathcal{B}(\widehat{\mathcal{U}}^N) &\leq \max \{C_{\mathcal{B}},C_0+\abs{f(-C_0)},C_0+\abs{f(C_0)},2 \norm{f'}_\infty,(2\norm{f'}_\infty)^{-1}\}. 
    \end{split}
\end{align}
\end{theorem}

\begin{proof}
The proof will proceed in the following way. We will consider a modification of scalar conservation law \eqref{eqn:original-scl} where both flux function and initial data of \eqref{eqn:original-scl} have been replaced by ReLU neural networks. As a first step, we will show that the solution $\widehat{u}$ of the modified scalar conservation law is a good approximation of the solution $u$ of \eqref{eqn:original-scl}. Next, we will prove that it is possible to approximate $\widehat{u}$ with a ReLU neural network. Combining these two results then proves the theorem. 

\textbf{Step 1. } Let $u(T,\cdot,y)$ be the solution to \eqref{eqn:original-scl} for some $y\in Y$ at time $t=T$. By the maximum principle \cite[Thm. 2.14]{holden2015front}, it holds that 
\begin{equation}\label{eqn:max-principle}
        \norm{u(T,\cdot,y)}_{L^{\infty}(D)}\leq \norm{u_0(\cdot,y)}_{L^{\infty}(D)}
\end{equation}
since $u_0(\cdot,y)\in L^{\infty}(D)\cap BV(D)$.
As $u_0\in L^{\infty}(D\times Y)$ (A2) we find that also $u(T)\in L^{\infty}(D\times Y)$. This allows us to restrict the domain of the flux function $f:\mathbb{R}\to\mathbb{R}$ to the compact interval $[-C_0,C_0]$, where $C_0$ is defined in (A3). By assumption (A1) and Lemma \ref{lem:flux-approximation}, for every $K\in\mathbb{N}$ there exists a ReLU neural network $\widehat{f}:=\widehat{f}_K:[-C_0,C_0]\to\mathbb{R}$ that satisfies the error bound
\begin{equation}\label{eqn:flux-approximation}
    \norm{f-\widehat{f}}_{\text{Lip}([-C_0,C_0];\mathbb{R})}\leq \frac{2C_0\norm{f''}_{L^\infty([-C_0,C_0])}}{K}.
\end{equation}
Furthermore, we can approximate $u_0(x,\cdot)$ for every $x\in D$ by a ReLU neural network $\widehat{u}_0^\epsilon (x,\cdot)$. Based on $\widehat{u}_0^\epsilon $, we define a function $\Tilde{u}_0^\epsilon $ such that $\Tilde{u}_0^\epsilon (\cdot,y)$ is piecewise constant for all $y\in Y$ and $\Tilde{u}_0^\epsilon (x,\cdot)$ can be realized by a ReLU neural network for all $x\in D$. For this reason we take $J\in\mathbb{N}$ and set $\Delta x = (b-a)/J$ and $x_{j-1/2} = a+j\Delta x$, $1 \leq j \leq J+1$. We then set 
\begin{align}
    &\Tilde{u}_0^\epsilon (x,\cdot) = \widehat{u}_0^\epsilon (x_j,\cdot), \qquad \text{for }x\in C_j := [x_{j-1/2},x_{j+1/2}], \quad 1 \leq j \leq J, \\
    &\overline{u_0}(x,\cdot) = u_0(x_j,\cdot), \qquad \text{for }x\in C_j, \quad 1 \leq j \leq J.
\end{align}
We will prove an upper bound for $\norm{u_0-\Tilde{u}_0^\epsilon }_{L^{\infty}(Y; L^1(D))}$. Using the triangle inequality, we obtain
\begin{align}\label{eqn:step1-part2}
    \begin{split}
\norm{u_0-\Tilde{u}_0^\epsilon }_{L^{\infty}(Y; L^1(D))} \leq \norm{u_0-\overline{u_0}}_{L^{\infty}(Y; L^1(D))} + \norm{\overline{u_0}-\Tilde{u}_0^\epsilon }_{L^{\infty}(Y; L^1(D))}. 
    \end{split}
\end{align}
We first prove that $\norm{u_0-\overline{u_0}}_{L^{\infty}(Y; L^1(D))}$ is small for large $J$. 
Rewriting this norm and performing a change of variables gives us
\begin{align}
    \begin{split}
&\norm{u_0-\overline{u_0}}_{L^{\infty}(Y; L^1(D))}\\ 
& \qquad = \sup_{y\in Y} \sum_j \int_{C_j} \vert u_0(x,y)-u_0(x_j,y) \vert dx\\
& \qquad \leq   \sup_{y\in Y}\int_{0}^{\frac{b-a}{2J}}\sum_j \left(\abs{u_0(x_{j}+t,y)-u_0(x_j,y)}+\abs{u_0(x_j,y)-u_0(x_{j-1/2}+t,y)}\right)dt\\
& \qquad \leq  \sup_{y\in  Y}\int_{0}^{\frac{b-a}{2J}} \text{TV}(u_0(\cdot,y)) dt = \frac{b-a}{2J}\sup_{y\in Y}\text{TV}(u_0(\cdot,y)).
    \end{split}
\end{align}
Now we still must find a bound for the second term in \eqref{eqn:step1-part2}. Note that
\begin{equation}
    \norm{\overline{u_0}(\cdot,y)-\Tilde{u}_0^\epsilon (\cdot,y)}_{L^1(D)} = \frac{b-a}{J} \sum_j \vert u_0(x_j,y) -\widehat{u}_0^\epsilon (x_j,y) \vert
\end{equation}
is nothing more than a numerical quadrature approximation of $\norm{u_0(\cdot,y)-\widehat{u}_0^\epsilon (\cdot,y)}_{L^1(D)}$. Since $\sup_{y\in Y} \mathrm{TV}(u_0(\cdot,y))<\infty$ (A2) and $\sup_{y\in Y,\epsilon>0} \mathrm{TV}(\widehat{u}_0^\epsilon (\cdot,y))<\infty$ (A3) it follows from the Koksma-Hlawka inequality \cite{caflisch1998monte} that 
\begin{equation}
    \left\vert \norm{u_0-\widehat{u}_0^\epsilon }_{L^{\infty}(Y; L^1(D))} - \norm{\overline{u_0}-\Tilde{u}_0^\epsilon }_{L^{\infty}(Y; L^1(D))} \right\vert \leq \frac{4(b-a)\ctv}{J},
\end{equation}
where $\ctv=\sup_{y\in Y,\epsilon>0}\max\{\text{TV}(u_0(\cdot,y),\text{TV}(\widehat{u}_0^\epsilon (\cdot,y)\}$.  From assumptions (A2) and (A3) it follows that $\ctv<\infty$. Since we know by (A3) that for any $\epsilon>0$ we can choose $\widehat{u}_0^\epsilon $ such that $\norm{u_0-\widehat{u}_0^\epsilon }_{L^{\infty}(Y; L^1(D))}\leq \epsilon$, we have now successfully bounded the second term in \eqref{eqn:step1-part2}. In summary, we have that 
\begin{equation}
    \norm{u_0-\Tilde{u}_0^\epsilon }_{L^{\infty}(Y; L^1(D))}\leq \frac{9}{2}\ctv\frac{b-a}{J} + \epsilon. 
\end{equation}

Now denote by $\widehat{u}^\epsilon$ the solution to the following modification of scalar conservation law \eqref{eqn:original-scl}: 
\begin{equation}\label{eqn:modified-scl}
    \begin{cases}
    \partial_t \widehat{u}^\epsilon(t,x,y) + \partial_x \widehat{f}(\widehat{u}^\epsilon(t,x,y)) = 0,\\
    \widehat{u}^\epsilon(0,x,y) = \Tilde{u}_0^\epsilon (x,y),
    \end{cases}
\end{equation}
for all $x\in D=[a,b]$, $y\in Y=[0,1]^d$ and $t\in[0,T]$. Note that \eqref{eqn:modified-scl} is well-defined because the maximum principle ensures again that 
\begin{equation}
\sup_{x\in D, y \in Y, \epsilon>0} |\widehat{u}^\epsilon(T,x,y)|\leq C_0, 
\end{equation}
where $C_0$ is as in (A3). 
By \cite[Thm. 4.3]{holden2015front} and \eqref{eqn:flux-approximation}, it then holds that
\begin{align}\label{eqn:step1-result}
    \begin{split}
         \norm{u-\widehat{u}^\epsilon}_{L^{\infty}([0,T] \times Y; L^1(D))} 
         &\leq  \norm{u_0-\Tilde{u}_0^\epsilon }_{L^{\infty}(Y; L^1(D))} + \ctv T \norm{f-\widehat{f}}_{\text{Lip}([-C_0,C_0];\mathbb{R})} \\
         &\leq \frac{9\ctv(b-a)}{2J} + \epsilon + \frac{\ctv TC_0}{K}\norm{f''}_{L^\infty([-C_0,C_0])}.
    \end{split}
\end{align}

\textbf{Step 2. } We now proceed to show that $\widehat{u}^\epsilon$ can be approximated using a ReLU neural network. We do this by emulating the Lax-Friedrichs scheme. For this purpose, we discretize $[0,T]\times D$ in time and space: let $N\in\mathbb{N}$ and define $\Delta t = T/N$ and $t^n = n\Delta t$, $0\leq n \leq N$. We then set $\Delta x = F\Delta t = (b-a)/J$ for $F=\norm{f'}_{L^\infty([-C_0,C_0])}$ and $J\in \mathbb{N}$ such that the CFL condition is satisfied. Note that this is the same $J$ as in step 1. Furthermore we define $x_{j-1/2} = a+j\Delta x$, $1 \leq j \leq J+1$. For every $y$ we can then define the cell averages
\begin{equation}
\widehat{U}^0_j(y) = \frac{1}{\Delta x}\int_{x_{j-1/2}}^{x_{j+1/2}} \Tilde{u}_0^\epsilon (x,y) dx = \Tilde{u}_0^\epsilon (x_j,y) = \widehat{u}_0^\epsilon (x_j,y), \qquad 1\leq j \leq J. 
\end{equation}
Thanks to our choice of initial data $\Tilde{u}_0^\epsilon $, we can explicitly calculate the cell averages, instead of approximating them. Next, we will approximate $\widehat{u}^\epsilon$ at time $t=T$ using a finite volume scheme. We opt for the Lax-Friedrichs scheme because of its simple form and suitable theoretical properties. The Lax-Friedrichs approximation of $\frac{1}{\Delta x}\int_{x_{j-1/2}}^{x_{j+1/2}}\widehat{u}^\epsilon(t^{n+1},x,y)dx$ is recursively defined as
\begin{equation}\label{eqn:lxf}
    \widehat{U}^{n+1}_j(y) = \frac{\widehat{U}^n_{j+1}(y)+\widehat{U}^n_{j-1}(y)}{2} - \frac{\Delta t}{2 \Delta x}(\widehat{f}(\widehat{U}^n_{j+1}(y))-\widehat{f}(\widehat{U}^n_{j-1}(y))). 
\end{equation}
Now recall that $\widehat{u}_0^\epsilon (x_j,y)$, and therefore also $\widehat{U}^0_j(y)$, can be written as a ReLU neural network. Furthermore, $\widehat{U}^{n+1}_j(y)$ can be written as a ReLU DNN, as $\widehat{f}$ itself is a ReLU DNN. It is therefore possible to obtain a ReLU DNN that maps $y$ to $(\widehat{U}^{N}_1(y), \ldots, \widehat{U}^{N}_J(y))$. This network is visualized in Figure \ref{fig:flowchart}. 

Let $\widehat{\mathcal{U}}^N(\cdot,y)$ be the piecewise constant function associated to to the constructed DNN, i.e. $\widehat{\mathcal{U}}^N(x,y)= \widehat{U}^{N}_j(y)$ for $x\in(x_{j-1/2},x_{j+1/2})$ (cf. Definition \ref{def:sol-associated-dnn}). The accuracy of this approximation is quantified by an error estimate due to Kuznetsov (Lemma \ref{lem:LxF-rate}). For every $y$, it holds that
\begin{equation}\label{eqn:step2-result}
    \sup_{y\in Y}\norm{\widehat{u}^\epsilon(T,\cdot,y)-\widehat{\mathcal{U}}^N(\cdot,y)}_{L^{1}(D)} \leq   \frac{31\ctv T\left(1+F\right)^{2}}{\sqrt{N}}, 
\end{equation}
where $\ctv$ was defined in Step 1 and $F=\norm{f'}_{L^\infty([-C_0,C_0])}$. Note that in practice the observed convergence rate is generally higher than $1/2$.

\textbf{Step 3. } We now combine the results from the previous steps. Set $\epsilon = 1/\sqrt{N}$ and $K = \left\lceil\sqrt{N}\right\rceil$ such that $K\leq 2\sqrt{N}$. 
Using \eqref{eqn:step1-result} and \eqref{eqn:step2-result} and by recalling that $(b-a)N=FTJ$ we obtain
\begin{align}
    \begin{split}
        &\sup_{y\in Y}\norm{u(T,\cdot,y)-\widehat{\mathcal{U}}^N(\cdot,y)}_{L^{1}(D)}\\ &\quad \leq \sup_{y\in Y}\norm{u(T,\cdot,y)-\widehat{u}^\epsilon(T,\cdot,y)}_{L^{1}(D)} + \sup_{y\in Y}\norm{\widehat{u}^\epsilon(T,\cdot,y)-\widehat{\mathcal{U}}^N(\cdot,y)}_{L^{1}(D)}\\
        &\quad\leq \frac{9\ctv FT}{2N} + \epsilon + \frac{2\ctv TC_0}{K}\norm{f''}_{L^\infty([-C_0,C_0])} + \frac{31\ctv T\left(1+F\right)^{2}}{\sqrt{N}} \\
        &\quad\leq  \frac{2\ctv T\left(C_0\norm{f''}_{L^\infty([-C_0,C_0])}+18\left(1+F\right)^{2}\right)+1}{\sqrt{N}}.
    \end{split}
\end{align}
This proves \eqref{eqn:expressivity-dnn-scl}. 

\textbf{Step 4. } Lastly, we calculate the depth and the number of non-zero weights and biases of the ReLU neural network $\widehat{\mathcal{U}}:=\widehat{\mathcal{U}}^N$, where we identify the constructed ReLU network with vectorial output $(\widehat{U}^{N}_1(y), \ldots, \widehat{U}^{N}_J(y))$ with the associated piecewise constant function (cf. Definition \ref{def:sol-associated-dnn}). 
Given that there are $N$ time steps, and that the network $\widehat{f}$ has only one hidden layer, the total network can be seen to have $\mathcal{L}(\widehat{u}_0^\epsilon)+N$ hidden layers. Similarly, the maximum width of the network can be seen to be $J\cdot\max\{2+K,\mathcal{W}(\widehat{u}_0 ^\epsilon)\}\leq 2J\cdot\max\{1+\sqrt{N},\mathcal{W}(\widehat{u}_0^\epsilon)\}$. Finally, the network essentially consists of $J\cdot N$ identical subnetworks, which we denote by $\widehat{\mathcal{L}}$. Each of these subnetworks consists of two copies of $\widehat{f}$ and two copies of a ReLU neural network (of which the depth matches that of $\widehat{f}$) that realizes the identity function. In both cases, two copies correspond to the evaluations in $\widehat{U}^n_{j-1}(y)$ and $\widehat{U}^n_{j+1}(y)$, cf. \eqref{eqn:lxf}. It is also clear that the calculation of $(\widehat{U}^0_j(y))_j$ gives rise to $J$ networks with each $O(d^{\sigma_{\mathcal{M}}}\epsilon^{-\eta_{\mathcal{M}}})$ non-zero weights and biases. Given that $\widehat{f}$ consists of $O(K)$ non-zero weights and biases, the total network has $O\left(J \mathcal{M}(\widehat{u}_0^\epsilon)+JNK\right)$ non-zero weights and biases. Finally, using that $K\leq 2\sqrt{N}$, $(b-a)N=FTJ$ and assumption (A3) concludes the proof. 

\end{proof}

\begin{remark}
The proof of Theorem \ref{thm:scl-dnn-1d} does not critically depend on the use of the Lax-Friedrichs scheme and the ReLU activation function. In fact, the argument also works with many other activation functions and finite volume schemes for which the convergence rate can explicitly proven. The advantage of the Lax-Friedrichs scheme is that it easily can be written as a ReLU neural network and that the proof of its convergence rate is straightforward. 
\end{remark}

\begin{figure}

\tikzstyle{block} = [rectangle, draw, 
    text centered, rounded corners]
    \tikzstyle{block2} = [rectangle, draw, 
    text centered, rounded corners, minimum height=6em]
        \tikzstyle{block3} = [rectangle, draw, 
    text centered, rounded corners, minimum height=3em, minimum width = 2em]
            \tikzstyle{block4} = [rectangle, draw, 
    text centered, rounded corners, minimum height=9em]
\tikzstyle{line} = [draw, -latex']
\centering
\resizebox{0.8\textwidth}{!}{
\begin{tikzpicture}

\def\start{-1.5}

\node[block] at (-2,\start) (a1) {$y_1$};
\node[block] at (-1,\start) (a2) {$y_2$};
\node at (0,\start) {$\cdots$};
\node[block] at (1,\start) (a3) {$y_{d-1}$};
\node[block] at (2,\start) (a4) {$y_{d}$};

\node[block] at (-4,1) (b1) {$\widehat{U}^0_1(y)$};
\node[block] at (-2.5,1) (b2) {$\widehat{U}^0_2(y)$};
\node[block] at (-1,1) (b3) {$\widehat{U}^0_3(y)$};
\node at (0.25,1) (b7) {$\cdots$};
\node[block] at (1.5,1) (b4) {$\widehat{U}^0_{J-2}(y)$};
\node[block] at (3.25,1) (b5) {$\widehat{U}^0_{J-1}(y)$};
\node[block] at (4.75,1) (b6) {$\widehat{U}^0_{J}(y)$};

\node[block3] at (-2.5,2.5) (c1) {$\widehat{\mathcal{L}}$};
\node[block3] at (-1,2.5) (c2) {$\widehat{\mathcal{L}}$};
\node[block3] at (1.5,2.5) (c3) {$\widehat{\mathcal{L}}$};
\node[block3] at (3.25,2.5) (c4) {$\widehat{\mathcal{L}}$};
\node at (0.25,2.5) (c5) {$\cdots$};

\def\up{4}

\node[block] at (-4,\up) (d1) {$\widehat{U}^1_1(y)$};
\node[block] at (-2.5,\up) (d2) {$\widehat{U}^1_2(y)$};
\node[block] at (-1,\up) (d3) {$\widehat{U}^1_3(y)$};
\node at (0.25,\up) {$\cdots$};
\node[block] at (1.5,\up) (d4) {$\widehat{U}^1_{J-2}(y)$};
\node[block] at (3.25,\up) (d5) {$\widehat{U}^1_{J-1}(y)$};
\node[block] at (4.75,\up) (d6) {$\widehat{U}^1_{J}(y)$};

\def\up{5.25}

\node[block] at (-4,\up)  {$\widehat{U}^N_1(y)$};
\node[block] at (-2.5,\up) {$\widehat{U}^N_2(y)$};
\node[block] at (-1,\up)  {$\widehat{U}^N_3(y)$};
\node at (0.25,\up) {$\cdots$};
\node[block] at (1.5,\up)  {$\widehat{U}^N_{J-2}(y)$};
\node[block] at (3.25,\up)  {$\widehat{U}^N_{J-1}(y)$};
\node[block] at (4.75,\up)  {$\widehat{U}^N_{J}(y)$};

\def\up{4.75}

\node at (-4,\up)  {$\vdots$};
\node at (-2.5,\up) {$\vdots$};
\node at (-1,\up)  {$\vdots$};
\node at (0.25,\up) {$\vdots$};
\node at (1.5,\up)  {$\vdots$};
\node at (3.25,\up)  {$\vdots$};
\node at (4.75,\up)  {$\vdots$};

\path [line] (a1) -- (b1);
\path [line] (a1) -- (b2);
\path [line] (a1) -- (b3);
\path [line] (a1) -- (b4);
\path [line] (a1) -- (b5);
\path [line] (a1) -- (b6);

\path [line] (a2) -- (b1);
\path [line] (a2) -- (b2);
\path [line] (a2) -- (b3);
\path [line] (a2) -- (b4);
\path [line] (a2) -- (b5);
\path [line] (a2) -- (b6);

\path [line] (a3) -- (b1);
\path [line] (a3) -- (b2);
\path [line] (a3) -- (b3);
\path [line] (a3) -- (b4);
\path [line] (a3) -- (b5);
\path [line] (a3) -- (b6);

\path [line] (a4) -- (b1);
\path [line] (a4) -- (b2);
\path [line] (a4) -- (b3);
\path [line] (a4) -- (b4);
\path [line] (a4) -- (b5);
\path [line] (a4) -- (b6);

\path [line] (b1) -- (c1);
\path [line] (b3) -- (c1);

\path [line] (b2) -- (c2);
\path [line] (b7) -- (c2);

\path [line] (b7) -- (c3);
\path [line] (b5) -- (c3);

\path [line] (b4) -- (c4);
\path [line] (b6) -- (c4);

\path [line] (c1) -- (d2);
\path [line] (c2) -- (d3);
\path [line] (c3) -- (d4);
\path [line] (c4) -- (d5);

\def\ypos{1}
\def\xpos{8}

\node[block] at (\xpos,\ypos) (A1) {$\widehat{U}^n_{j-1}(y)$};
\node[block] at (\xpos+2,\ypos) (A2) {$\widehat{U}^n_{j+1}(y)$};

\node[block3] at (\xpos-0.5,\ypos+1.5) (B1) {$id$};
\node[block3] at (\xpos+0.5,\ypos+1.5) (B2) {$\widehat{f}$};
\node[block3] at (\xpos+1.5,\ypos+1.5) (B3) {$\widehat{f}$};
\node[block3] at (\xpos+2.5,\ypos+1.5) (B4) {$id$};

\node[block] at (\xpos+1,\ypos+3) (C) {$\widehat{U}^{n+1}_{j}(y)$};

\path [line] (A1) -- (B1);
\path [line] (A1) -- (B2);
\path [line] (A2) -- (B3);
\path [line] (A2) -- (B4);
\path [line] (B1) -- (C);
\path [line] (B2) -- (C);
\path [line] (B3) -- (C);
\path [line] (B4) -- (C);

\draw[rounded corners] (\xpos-1, \ypos+3.5) rectangle (\xpos+3, \ypos-0.5);

\node at (\xpos+1,\ypos-1)  {$\widehat{\mathcal{L}}$};


\end{tikzpicture}
}
\caption{Flowchart of the ReLU neural networks $\widehat{\mathcal{U}}^N$ and $\widehat{\mathcal{L}}$ for a fixed flux. }
\label{fig:flowchart}
\end{figure}
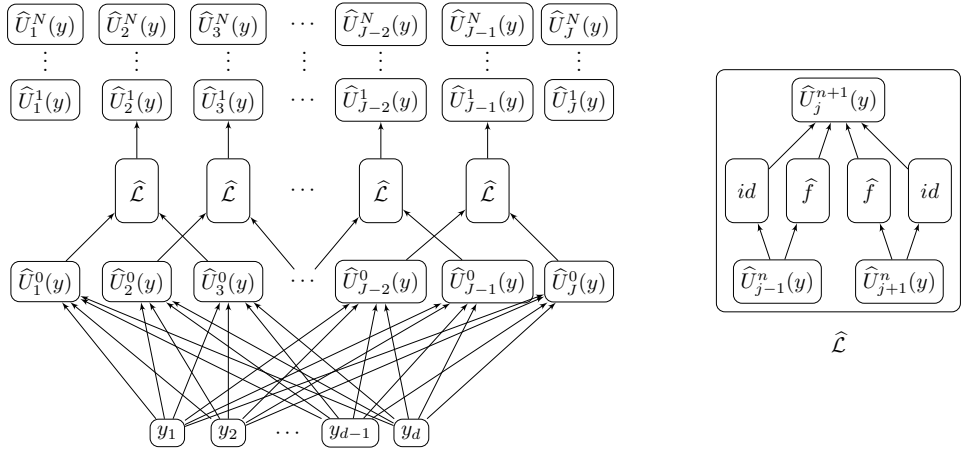

The bounds of Theorem \ref{thm:scl-dnn-1d} are particularly simple in the setting of Example \ref{ex:kle}. The following corollary shows that the connectivity and network width only depend linearly on the parameter space dimension. 

\begin{corollary}[Theorem \ref{thm:scl-dnn-1d} for Karhunen-Loève expansion]\label{cor:scl-dnn-1d-kle}
Let $T>0$, $d\in\mathbb{N}$, $a,b\in\mathbb{R}$ with $a<b$ and $D=[a,b]$. Let $C_0 = \norm{u_0}_{L^{\infty}(D\times Y)}$ and $\ctv=\sup_{y\in Y}\text{TV}(u_0(\cdot,y))$. For $y\in Y=[0,1]^d$, denote by $u(T,\cdot,y)$ the solution at time $t=T$ of \eqref{eqn:original-scl} with a flux $f:[-C_0,C_0]\to\mathbb{R}$ that satisfies assumption (A1) and initial condition $u_0(\cdot,y)$ given by a Karhunen-Loève expansion (cf. Example \ref{ex:kle}). Then for every $N\in\mathbb{N}$ there exists a ReLU neural network such that the associated solution map $\widehat{\mathcal{U}}^N:Y\to\mathbb{R}$ (cf. Defintion \ref{def:sol-associated-dnn}) satisfies the error bound
\begin{equation}\label{eqn:expressivity-dnn-scl-kle}
\sup_{y\in Y}\norm{u(T,\cdot,y)-\widehat{\mathcal{U}}^N(\cdot,y)}_{L^1(D)} \leq
    \frac{2\ctv T\left(C_0\norm{f''}_\infty+18\left(1+\norm{f'}_\infty\right)^{2}\right)}{\sqrt{N}}.
\end{equation}
In addition, it holds that for $N\in\mathbb{N}$,
\begin{align}
    \begin{split}
        \mathcal{M}(\widehat{\mathcal{U}}^N) &= O\left(dN+N^{5/2}\right), \\
        \mathcal{L}(\widehat{\mathcal{U}}^N) &\leq N+1, \\
        \mathcal{W}(\widehat{\mathcal{U}}^N) &\leq \max\left\{1+d,\frac{2(b-a)N^{3/2}}{T\norm{f'}_\infty}\right\} , \\ 
         \mathcal{B}(\widehat{\mathcal{U}}^N) &\leq \max \{C_{\mathcal{B}},C_0+\abs{f(-C_0)},C_0+\abs{f(C_0)},2 \norm{f'}_\infty,(2\norm{f'}_\infty)^{-1}\}.
    \end{split}
\end{align}
\end{corollary}

\subsection{Parametric flux and initial data}\label{sec:flux}

We expand the results from the previous section to one-dimensional scalar conservation laws with parametric initial condition and a parametric flux, 

\begin{align}\label{eqn:original-scl-flux}
\begin{split}
    \begin{cases}\partial_t u(t,x,y,z) + \partial_x f(u(t,x,y),z) = 0,\\
    u(0,x,y,z) = u_0(x,y),\end{cases}
\end{split}
\end{align}
for all $x\in D=[a,b]$, $y\in Y=[0,1]^d$, $z\in Z=[0,1]^\s$ and $t\in[0,T]$. In this notation, the initial condition is parametrized by the vector $y$ and the flux function by the vector $z$. To prove our approximation result we need to slightly adapt assumptions (A1-3) we made in the previous section. In the current setting we assume: 
\begin{itemize}
    \item[\textit{(B1)}] We replace (A1) by the assumption that the flux is in the form of a truncated Karhunen-Loève expansion (cf. Example \ref{ex:kle}),
\begin{equation*}\label{eqn:kl-flux}
    f(u,z) \approx \overline{f}(u)+\sum_{i=1}^\s \sqrt{\lambda_i}z_i\varphi_i(u),
\end{equation*}
where $\overline{f}\in C^2(\mathbb{R};\mathbb{R})$ and $\varphi_i$ are some smooth basis functions. 
    
    \item[\textit{(B2-3)}] We take over assumptions (A2) and (A3), including the definition of $C_0$ and $\ctv$.
    \item[\textit{(B4)}] We make the assumption that there exist constants $C_f>1$ and $\sigma_f>0$ such that it holds that
    \begin{equation*}
        \norm{\partial_u f(u,z)}_{L^\infty([-C_0,C_0])} \leq C_f \s^{\sigma_f}
    \end{equation*}
    and that
    \begin{equation*}
        \max_i\sqrt{\lambda_i}\left(C_0\norm{\varphi_i''}_\infty+\frac{\norm{\varphi_i}_\infty+1}{4}\norm{\varphi_i'}_\infty\right) \leq C_f \s^{2\sigma_f}, 
    \end{equation*}
    where $\norm{\cdot}_\infty := \norm{\cdot}_{L^\infty([-C_0,C_0])}$. 
\end{itemize}
Assumption (B4) is related to the decay of the Karhunen-Loève eigenvalues. We will see in the numerical experiments that these conditions are generally easily satisfied. 
We can then obtain the following result. 

\begin{theorem}[Extension of Theorem \ref{thm:scl-dnn-1d}  to uncertain flux]\label{thm:scl-dnn-flux-1d}
Let $T>0$, $d,\s\in\mathbb{N}$, $a,b\in\mathbb{R}$ with $a<b$ and $D=[a,b]$. For $y\in Y=[0,1]^d$ and $z\in Z=[0,1]^\s$, denote by $u(T,\cdot,y,z)$ the solution at time $t=T$ of \eqref{eqn:original-scl} with flux $f(\cdot,z)$ as in (B1) and initial condition $u_0(\cdot,y)$ that satisfies assumptions (A2) and (A3). Then for every $N\in\mathbb{N}$ there exists a ReLU neural network $\widehat{\mathcal{U}}^N$ that satisfies the error bound
\begin{equation}\label{eqn:expressivity-flux-dnn-scl}
    \sup_{y\in Y}\norm{u(T,\cdot,y)-\widehat{\mathcal{U}}^N(\cdot,y)}_{L^1(D)} \leq  \frac{2\ctv T\left(C_0\norm{\overline{f}''}_\infty +19\left(1+C_f \s^{\sigma_f}\right)^{2}\right)+1}{\sqrt{N}}.
\end{equation}
Let In addition, it holds that that there exist a constant $\gamma(\sigma_f,\eta_{\mathcal{L}},\eta_{\mathcal{W}},\eta_{\mathcal{M}})\geq 0$ with the property that $\gamma=0$ if $\sigma_f=0$ such that for $N\in\mathbb{N}$ it holds that
\begin{align}
    \begin{split}
        \mathcal{M}(\widehat{\mathcal{U}}^N) &= O \left( d^{\sigma_{\mathcal{M}}}\s^{\gamma}N^{\eta_{\mathcal{M}}/2}+s^{2+\gamma}N^{5/2}\right), \\
        \mathcal{L}(\widehat{\mathcal{U}}^N) &=O\left( d^{\sigma_{\mathcal{L}}}\s^{\gamma}N^{\eta_{\mathcal{L}}/2}+s^\gamma N\ln(sN)\right), \\
        \mathcal{W}(\widehat{\mathcal{U}}^N) &=O\left( d^{\sigma_{\mathcal{W}}}\s^{\gamma}N^{\eta_{\mathcal{W}}/2}+s^{2+\gamma}N^{3/2}\right), \\ 
         \mathcal{B}(\widehat{\mathcal{U}}^N) &=O(1). 
    \end{split}
\end{align}
\end{theorem}

\begin{proof}
The proof is essentially the same as that of Theorem \ref{thm:scl-dnn-1d}, only the approximation of $f$ requires more attention, as well as determining the size of the network.

\textbf{Step 1. } We start by proving the equivalent of \eqref{eqn:flux-approximation} in the proof Theorem \ref{thm:scl-dnn-1d}. Using the triangle inequality, it suffices to treat each term in the flux of (B1) separately. We approximate $\overline{f}$ by the network from Lemma \ref{lem:flux-approximation} using parameter $J\leftarrow K_1 = \lceil\sqrt{N^*}\rceil$, where we denote by $N^*$ the number of time steps. For every $i$, we approximate $\varphi_i$ by the network from Lemma \ref{lem:flux-approximation} using parameter $J\leftarrow K_2 = sK_1$. Finally, we approximate the multiplication operators using Lemma \ref{lem:mult-lipschitz} with parameters $m\leftarrow \log_2(s\sqrt{N^*})$, $M\leftarrow 1$ and $N\leftarrow \norm{\varphi_i}_\infty$. This then gives us the existence of a ReLU neural network $\widehat{f}$ that satisfies the error bound
\begin{align}\label{eqn:random-flux-approximation}
    \begin{split}
& \sup_{z\in Z}\norm{f(\cdot,z)-\widehat{f}(\cdot,z)}_{\text{Lip}([-C_0,C_0];\mathbb{R})}\\
&\quad \leq \norm{\overline{f}-\widehat{\overline{f}}}_{\text{Lip}([-C_0,C_0];\mathbb{R})}+\sum_{i=1}^s \sqrt{\lambda_i}\norm{\varphi_i-\widehat{\varphi}_i}_{\text{Lip}([-C_0,C_0];\mathbb{R})}\\
&\qquad + \sum_{i=1}^s \sqrt{\lambda_i} \sup_{z\in Z}\norm{\times(z_i,\cdot)-\widehat{\times}_m(z_i,\cdot)}_{\text{Lip}([-\norm{\varphi_i}_\infty,\norm{\varphi_i}_\infty];\mathbb{R})}\norm{\varphi_i}_{\text{Lip}([-C_0,C_0];\mathbb{R})}\\
&\quad \leq \frac{2C_0\norm{\overline{f} ''}_\infty}{K_1} + s \max_i\sqrt{\lambda_i}\left(\frac{2C_0\norm{\varphi_i''}_\infty}{K_2}+\frac{(\norm{\varphi_i}_\infty+1)\norm{\varphi_i'}_\infty}{2^{m+1}}\right)\\
&\quad \leq \frac{4C_0\norm{\overline{f}''}_\infty + \max_i\sqrt{\lambda_i}\left(4C_0\norm{\varphi_i''}_\infty+(\norm{\varphi_i}_\infty+1)\norm{\varphi_i'}_\infty\right)}{2\sqrt{N^*}} \\
&\quad \leq \frac{2C_0 \norm{\overline{f}''}_\infty + 2C_f \s^{2\sigma_f}}{\sqrt{N^*}} 
\leq \frac{2C_0 \norm{\overline{f}''}_\infty + 2(1+C_f)^2}{\sqrt{N}},
    \end{split}
\end{align}
where we used assumption (B4) and where we defined $N=s^{-4\sigma_f}N^*$ in order to make the error bound independent of $\s$. The existence of the network $\widehat{\mathcal{U}}^{N^*}$ and error bound \eqref{eqn:expressivity-flux-dnn-scl} then follow mutatis mutandis from the steps described in the proof of Theorem \ref{thm:scl-dnn-1d}. 

\textbf{Step 2. } We now determine the size of $\widehat{\mathcal{U}}^{N^*}$, cf. step 3 of the proof of Theorem \ref{thm:scl-dnn-1d}. A visualization can be found in Figure \ref{fig:flowchart-2}. Using the notation from the proof of Theorem \ref{thm:scl-dnn-1d}, the calculation of $(\widehat{U}^0_j(y))_j$ gives rise to $J^*$ parallel subnetworks of depth $C_{\mathcal{L}}d^{\sigma_{\mathcal{L}}}(N^*)^{\eta_{\mathcal{L}}/2}$, width $C_{\mathcal{W}}d^{\sigma_{\mathcal{W}}}(N^*)^{\eta_{\mathcal{W}}/2}$ and $C_{\mathcal{M}}d^{\sigma_{\mathcal{M}}}(N^*)^{\eta_{\mathcal{M}}/2}$ non-zero weights and biases. The second part of the network consists of $J^*\cdot N^*$ copies of the network $\widehat{f}$. The following holds: 
\begin{align}
    \begin{split}
        \mathcal{M}(\widehat{f}) &= \mathcal{M}\left(\widehat{\overline{f}}\right) + \s(\mathcal{M}\left(\widehat{\varphi}\right)+\mathcal{M}\left(\widehat{\times}\right))  = O\left(\sqrt{N^*}+\s^2\sqrt{N^*}+\s\log_2(\s\sqrt{N^*})\right) = O(\s^2\sqrt{N^*}), \\
        \mathcal{L}(\widehat{f}) &= \max\left\{\mathcal{L}\left(\widehat{\overline{f}}\right),\mathcal{L}\left(\widehat{\varphi}\right)+\mathcal{L}\left(\widehat{\times}\right)\right\} = O(\log_2(\s\sqrt{N^*})),\\
        \mathcal{W}(\widehat{f}) &= \mathcal{W}\left(\widehat{\overline{f}}\right) + s(\mathcal{W}\left(\widehat{\varphi}\right)+\mathcal{W}\left(\widehat{\times}\right)) = O(\s^2\sqrt{N^*}).
    \end{split}
\end{align}
Concatenating the two aforementioned parts of the network gives therefore rise to $L:=N^*\cdot\mathcal{L}(\widehat{f})$ layers. Finally, a third part of the network makes the value of the $\s$-dimensional vector $z$ accessible to all layers of the other two parts of the network. This can be done using a subnetwork of width $\s$, depth $L$ and connectivity $O(\s L)$. Adding the contributions from all three subnetworks, one obtains the claimed complexities of the theorem. To make the results presentable, we only report the complexities in the theorem. Note however that exact bounds can very easily be obtained from our calculations. 
\end{proof}

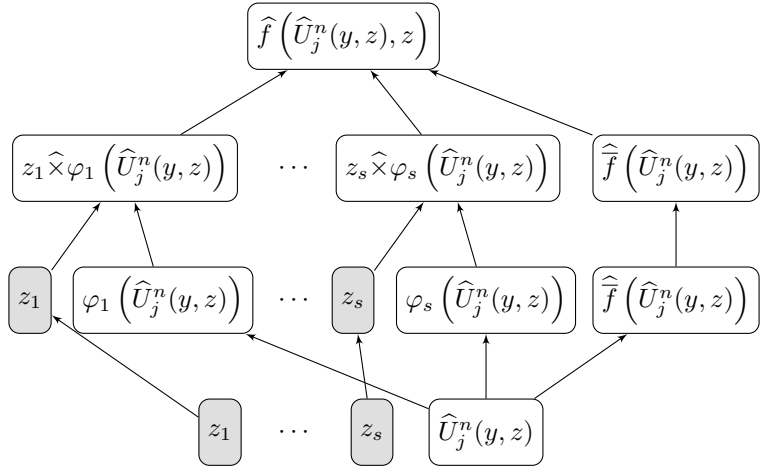
\begin{figure}
\tikzstyle{block} = [rectangle, draw, minimum height=2.5em, 
    text centered, rounded corners]
\tikzstyle{block2} = [rectangle, draw, minimum height=2.5em, 
    text centered, rounded corners,fill={rgb:black,1;white,7}]

\tikzstyle{line} = [draw, -latex']
\centering

\begin{tikzpicture}

\def\start{-0.75}

\node[block2] at (-0.5+2,\start) (a1) {$z_1$};
\node[block2] at (1.5+2,\start) (a2) {$z_\s$};
\node at (0.5+2,\start) {$\cdots$};
\node[block] at (3+2,\start) (a3) {$\widehat{U}^n_j(y,z)$};

\node[block2] at (-1,1) (b1) {$z_1$};
\node[block] at (0.75,1) (b2) {$\varphi_1\left(\widehat{U}^n_j(y,z)\right)$};
\node at (2.5,1) (b7) {$\cdots$};
\node[block2] at (3.25,1) (b4) {$z_\s$};
\node[block] at (5,1) (b5) {$\varphi_\s\left(\widehat{U}^n_j(y,z)\right)$};
\node[block] at (7.5,1) (b6) {$\widehat{\overline{f}}\left(\widehat{U}^n_j(y,z)\right)$};

\node[block] at (+0.25,2-\start) (c1) {$z_1 \widehat{\times} \varphi_1\left(\widehat{U}^n_j(y,z)\right)$};
\node at (2.5,2-\start)  {$\cdots$};
\node[block] at (4.5,2-\start) (c2) {$z_\s \widehat{\times} \varphi_\s\left(\widehat{U}^n_j(y,z)\right)$};
\node[block] at (7.5,2-\start) (c3) {$\widehat{\overline{f}}\left(\widehat{U}^n_j(y,z)\right)$};

\node[block] at (3.125,3-2*\start) (d) {$\widehat{f}\left(\widehat{U}^n_j(y,z),z\right)$};

\path [line] (a1) -- (b1);
\path [line] (a2) -- (b4);
\path [line] (a3) -- (b2);
\path [line] (a3) -- (b5);
\path [line] (a3) -- (b6);

\path [line] (b1) -- (c1);
\path [line] (b2) -- (c1);
\path [line] (b4) -- (c2);
\path [line] (b5) -- (c2);
\path [line] (b6) -- (c3);

\path [line] (c1) -- (d);
\path [line] (c2) -- (d);
\path [line] (c3) -- (d);

\end{tikzpicture}
    \caption{Flowchart of the ReLU neural network $\widehat{f}$ for uncertain flux functions. Gray neurons are shared over all $j$. }
    \label{fig:flowchart-2}
\end{figure}

\subsection{Extensions}
In what follows, we present a number of ways in which Theorem \ref{thm:scl-dnn-1d} and Theorem \ref{thm:scl-dnn-flux-1d} can be extended. For simplicity, we will only prove these results for a fixed flux function. The extension to a parametric flux function is completely analogous to the proof of Theorem \ref{thm:scl-dnn-flux-1d}.
Moreover, all results can also be generalized to systems of conservation laws, as long as it can be proven that some finite volume scheme, such as the Lax-Friedrichs scheme, converges to the true solution at a known convergence rate. We consider extensions to different neural network architectures (Section \ref{sec:rnn}), different kinds of approximation (observables in Section \ref{sec:obs} and space-time in Section \ref{sec:space-time}) as well as multi-dimensional scalar conservation laws (Section \ref{sec:multi-d}). 

\subsubsection{Residual and recurrent neural networks}\label{sec:rnn}

The network that has been constructed in the proof of Theorem \ref{thm:scl-dnn-1d} is very sparse and consists of subnetworks that are many times repeated. One possibility to reduce the size of the network is the introduction of skip connections, i.e. connections between non-consecutive hidden layers. These kind of networks are referred to as residual neural networks. Skip connections can most notably be used to simplify the neural network representing \eqref{eqn:lxf} (see also $\mathcal{L}$ in Figure \ref{fig:flowchart}). Note that there is also a strong connection between the structure of the constructed network and that of a recurrent neural network: the network to calculate $ (\widehat{U}^{n+1}_j(y))_j$ from $ (\widehat{U}^{n}_j(y))_j$ is independent from $n$. Casting a part of the constructed network as a recurrent network is thus another way to drastically reduce the size of the network. 

\subsubsection{Observables}\label{sec:obs}

In some applications one is not necessarily interested in the full solution of the conservation law, but rather in the image of the solution under some functional or so-called \textit{observable} or \textit{quantity of interest}. These can be expressed in the generic form 
\begin{equation}
    L_T(y,u) = \int_{D} \phi(x) g(u(T,x,y))dx, 
\end{equation}
where $\phi:D\to \mathbb{R}$ and $g:\mathbb{R}\to\mathbb{R}$ are suitable test functions. This allows us to define the \textit{parameters-to-observable} map
\begin{equation}
    \mathcal{L}_T(y):Y\to\mathbb{R}: y\mapsto L_T(y,u). 
\end{equation}
Provided that $\phi$ and $g$ are sufficiently regular, it follows from Theorem \ref{thm:scl-dnn-1d} that $\mathcal{L}_T$ can also be approximated using DNNs of which the size only depends linearly on the input dimension $d$.

\subsubsection{Space-time approximations}\label{sec:space-time}

Theorem \ref{thm:scl-dnn-1d} approximates the solution of \eqref{eqn:original-scl} on a grid at a fixed time $T>0$. It is natural to ask the question whether this result can be extended to space-time $D\times [0,T]$. The following corollary shows that this is indeed possible, if one allows the network weights to grow with increasing $N$. 

\begin{corollary}[Space-time extension of Theorem \ref{thm:scl-dnn-1d}]\label{cor:scl-dnn-1d-spacetime}Let $T>0$, $d\in\mathbb{N}$, $a,b\in\mathbb{R}$ with $a<b$ and $D=[a,b]$. For $y\in Y=[0,1]^d$, denote by $u(t,\cdot,y)$ the solution at time $t$ of \eqref{eqn:original-scl} with a flux $f:[-C_0,C_0]\to\mathbb{R}$ that satisfies assumption (A1) and initial condition $u_0(\cdot,y)$ that satisfies assumptions (A2) and (A3). Then for every $N\in\mathbb{N}$ there exists a ReLU neural network $\widehat{U}^N:[0,T]\times D \times Y\to \mathbb{R}$ that satisfies the error bound
\begin{equation}\label{eqn:expressivity-dnn-scl-spacetime}
    \sup_{y\in Y}\norm{u(t,\cdot,y)-\widehat{U}^N(t,\cdot,y)}_{L^1(D)} \leq  \frac{2\ctv T\left(C_0\norm{f''}_\infty+21\left(1+\norm{f'}_\infty\right)^{2}\right)+1}{\sqrt{N}}
\end{equation}
for all $t\in[0,T]$. In addition, it holds that that for $N\in\mathbb{N}$,
\begin{align}
    \begin{split}
        \mathcal{M}(\widehat{U}^N) &= O \left( d^{\sigma_{\mathcal{M}}}N^{1+\eta_{\mathcal{M}}/2}+N^{5/2}\right), \\
        \mathcal{L}(\widehat{U}^N) &=O( d^{\sigma_{\mathcal{L}}}N^{\eta_{\mathcal{L}}/2}+N), \\
        \mathcal{W}(\widehat{U}^N), &=O( d^{\sigma_{\mathcal{W}}}N^{1+\eta_{\mathcal{W}}/2}+N^{3/2}), \\ 
         \mathcal{B}(\widehat{U}^N) & = O(N). 
    \end{split}
\end{align}
\end{corollary}
\begin{proof}
In the notation of the proof of Theorem \ref{thm:scl-dnn-1d}, recall that $\widehat{U}^n_j(y)$ was the approximation we obtained for $u(t^n,x_j,y)$. We thus have access to an approximation to $u(\cdot,\cdot,y)$ on a grid in space-time. Our goal is to approximate $u(t,x,y)$ by some $\widehat{U}(t,x,y)$ for all $(t,x)$ based on $\widehat{U}^n_j(y)$. First, we define
\begin{equation}
    \alpha^n(t) = \max\left\{\min\left\{\frac{t^n-t}{\Delta t},1\right\},0\right\}\quad \text{and}\quad \beta_j(x) = \max\left\{\min\left\{\frac{x_j-x}{\Delta x},1\right\},0\right\}. 
\end{equation}
It holds that $\alpha^n(t) = 1$ if $t\leq t^{n-1}$ and $\alpha^n(t) = 0$ if $t\geq t^{n}$ and similarly for $\beta^j(x)$. In addition, both quantities can be written as a ReLU network, e.g. $\alpha^n(t) = \sigma(1-\sigma(1-(t^n-t)/\Delta t))$, where $\sigma$ is the ReLU activation function. Then note that a good approximation of $u(t,x_{j+1/2},y)$ would be given by 
\begin{equation}\label{eqn:approx-without-star}
    \widehat{U}^0_j + \sum_{n=1}^N \alpha^n(t) \cdot (\widehat{U}^n_j(y)-\widehat{U}^{n-1}_j(y)).
\end{equation}
As the multiplication operator can not be exactly written as a ReLU neural network, an approximation is necessary. Appendix \ref{sec:star-operation} introduces a ReLU neural network, which we denote by $\star$, consisting of one hidden layer with two neurons that mimics the multiplication operator adequately in our setting. We proceed by replacing the multiplication $\cdot$ by $\star$ in \eqref{eqn:approx-without-star} and define
\begin{equation}\label{eqn:time-dnn}
    \widehat{U}_j(t,y) = \widehat{U}^0_j + \sum_{n=1}^N \alpha^n(t) \star (\widehat{U}^n_j(y)-\widehat{U}^{n-1}_j(y))
\end{equation}
and similarly
\begin{equation}
    \widehat{U}(t,x,y) = \widehat{U}(t,x_1,y) + \sum_{j=1}^J \beta_j(x) \star (\widehat{U}_j(t,y)-\widehat{U}_{j-1}(t,y))
\end{equation}
Now we calculate that for $t^{n-1}\leq t < t^n$,
\begin{equation}
    \Delta x \sum_j\abs{\widehat{U}_j(t,y)-\widehat{U}^{n-1}_j(y)} \leq \Delta x \sum_j\abs{\widehat{U}_j^n(y)-\widehat{U}^{n-1}_j(y)} \leq 2F \text{TV}(\widehat{u}_0^\epsilon )\Delta t
\end{equation}
by \eqref{eqn:lxf-lipschitz-time} in Appendix \ref{app:kuznetsov}, where $F=\norm{f'}_{L^\infty([-C_0,C_0])}$. In addition, it holds for $x_{j-1}\leq x < x_j$ and $t^{n-1}\leq t < t^n$ that
\begin{align}
\begin{split}
    \abs{\widehat{U}(t,x,y)-\widehat{U}_{j}(t,y)} &\leq \abs{\widehat{U}_{j-1}(t,y)-\widehat{U}_{j}(t,y)}\\
    &  \leq \abs{\widehat{U}_{j-1}(t,y)-\widehat{U}^{n-1}_{j-1}(y)}+\abs{\widehat{U}^{n-1}_{j-1}(y)-\widehat{U}^{n-1}_{j}(y)}+\abs{\widehat{U}^{n-1}_{j}(y)-\widehat{U}_{j}(t,y)}.
\end{split}
\end{align}
Since the Lax-Friedrichs method is a monotone, consistent and conservative scheme it satisfies Harten's lemma and is therefore total variation diminishing (TVD), meaning that
\begin{equation}
\begin{split}
    \sum_j \abs{\widehat{U}^{n-1}_{j-1}(y)-\widehat{U}^{n-1}_{j}(y)} &\leq \sum_j \abs{\widehat{U}^{0}_{j-1}(y)-\widehat{U}^{0}_{j}(y)}\\ &= \sum_j \abs{\widehat{u}_0^\epsilon(x_{j-1},y)-\widehat{u}_0^\epsilon(x_{j},y)} \leq \text{TV}(\widehat{u}_0^\epsilon(\cdot,y))
\end{split}
\end{equation}
This allows us to quantify the accuracy of our ReLU neural network approximation with respect to the cell-averages from the proof of Theorem \ref{thm:scl-dnn-1d}, where we use the CFL condition $\Delta x=F\Delta t$,
\begin{align}
    \begin{split}
        \norm{\widehat{U}(t,\cdot,y)-\widehat{\mathcal{U}}^n(\cdot,y)}_{L^1(D)} &\leq 4F \text{TV}(\widehat{u}_0^\epsilon(\cdot,y) )\Delta t + \text{TV}(\widehat{u}_0^\epsilon(\cdot,y) )\Delta x \\&= 5F \text{TV}(\widehat{u}_0^\epsilon(\cdot,y) )\Delta t \\ &\leq 6(1+F)^2 \text{TV}(\widehat{u}_0^\epsilon(\cdot,y) ) T N^{-1/2}.
    \end{split}
\end{align}
This error bound, together with Theorem \ref{thm:scl-dnn-1d}, proves  \eqref{eqn:expressivity-dnn-scl-spacetime}. 



Finally, we note that the complexity estimates are easy to deduct from the above construction and are the same as those of Theorem \ref{thm:scl-dnn-1d} (up to a constant). 
\end{proof}

\subsubsection{Multidimensional scalar conservation laws}\label{sec:multi-d}

We have so far only treated one-dimensional scalar conservation laws. In real life applications, the multidimensional case may be more prevalent. For $m\in\mathbb{N}$, the $m$-dimensional version of \eqref{eqn:original-scl} is given by
\begin{align}\label{eqn:multidim-scl}
\begin{split}
\begin{cases}
    \partial_t u(t,x,y) + \sum_{j=1}^m\partial_{x_j} f_j(u(t,x,y)) = 0,\\
    u(0,x,y) = u_0(x,y),
    \end{cases}
\end{split}
\end{align}
for all $x\in D=[a,b]^m$, $y\in Y=[0,1]^d$ and $t\in[0,T]$. The generalization of assumptions (A1)-(A3) is straightforward. In this setting, we obtain the following approximation result. 

\begin{corollary}[Extension of Theorem \ref{thm:scl-dnn-1d} to multidimensional grids] \label{cor:scl-dnn-multidim}
Let $T>0$, $d,m\in\mathbb{N}$, $a,b\in\mathbb{R}$ with $a<b$ and $D=[a,b]^m$. For $y\in Y=[0,1]^d$, denote by $u(T,\cdot,y)$ the solution at time $t=T$ of \eqref{eqn:original-scl} with flux functions $f_j:\mathbb{R}\to\mathbb{R}$, $1\leq j\leq m$, that satisfy assumption (A1) and initial condition $u_0$ that satisfies assumptions (A2) and (A3). Then for every $N\in\mathbb{N}$, there exist a ReLU neural network $\widehat{\mathcal{U}}^N$ 
that satisfies the error bound
\begin{equation}\label{eqn:expressivity-dnn-scl-multidim}
    \sup_{y\in Y}\norm{u(T,\cdot,y)-\widehat{\mathcal{U}}^N(\cdot,y)}_{L^1(D)} \leq \frac{2\ctv T\left(C_0\norm{f''}_\infty+18\left(1+\norm{f'}_\infty\right)^{2}\right)+1}{\sqrt{N}}.
t\end{equation}
In addition, it holds that
\begin{align}
    \begin{split}
        \mathcal{M}(\widehat{\mathcal{U}}^N) &= O \left( d^{\sigma_{\mathcal{M}}}N^{m+\eta_{\mathcal{M}}/2}+mN^{m+3/2}\right), \\
        \mathcal{L}(\widehat{\mathcal{U}}^N) &=O( d^{\sigma_{\mathcal{L}}}N^{\eta_{\mathcal{L}}/2}+N), \\
        \mathcal{W}(\widehat{\mathcal{U}}^N), &=O( d^{\sigma_{\mathcal{W}}}N^{m+\eta_{\mathcal{W}}/2}+mN^{m+1/2}), \\ 
         \mathcal{B}(\widehat{\mathcal{U}}^N) &=O(1). 
    \end{split}
\end{align}
\end{corollary}

\begin{proof}
The proof of the corollary only requires a small generalization of the proof of Theorem \ref{thm:scl-dnn-1d}. Step 1 still holds for multidimensional scalar conservation laws. We will numerically approximate the solution of \eqref{eqn:multidim-scl} using dimensional splitting methods, i.e. by solving one space direction at a time. For $1\leq i,k \leq m$, let $e_i\in\mathbb{R}^m$ be the vector that satisfies $(e_i)_k = \delta_{ik}$. The multidimensional update formula of the Lax-Friedrichs scheme \eqref{eqn:lxf} is then given by
\begin{equation}
    \widehat{U}^{n+1}_j(y) = \sum_{i=1}^m\left(\frac{\widehat{U}^n_{j-e_i}(y)+\widehat{U}^n_{j+e_i}(y)}{2} - \frac{\Delta t}{2 \Delta x_i}(\widehat{f}_i(\widehat{U}^n_{j+e_i}(y))-\widehat{f}_i(\widehat{U}^n_{j-e_i}(y)))\right), 
\end{equation}
where $j$ is an $m$-dimensional vector that indicates the cell location. The rest of step 2 can be taken over mutatis mutandis. It only remains to determine the size of the network. As only the connectivity and maximum width of the network depend on $m$, $\mathcal{L}(\widehat{\mathcal{U}}^N)$ and $\mathcal{B}(\widehat{\mathcal{U}}^N)$ are the same as in Theorem \ref{thm:scl-dnn-1d}. In the notation of the proof of Theorem \ref{thm:scl-dnn-1d}, we let $N$ be the number of time steps, such that there are now $O(N^m)$ grid points. The first explains the factor $O(N^m)$ in the first term in $\mathcal{M}(\widehat{\mathcal{U}}^N)$ and $\mathcal{W}(\widehat{\mathcal{U}}^N)$. For every $f_i$, the size of the approximating network is given by Lemma \ref{lem:flux-approximation}. Consequently, the second term $\mathcal{W}(\widehat{\mathcal{U}}^N)$ follows directly as there are $m$ flux functions of maximal width $2\sqrt{N}$ that need to be evaluated on $O(N^m)$ grid points. For $\mathcal{M}(\widehat{\mathcal{U}}^N)$, the result is obtained by multiplying with the number of time steps. This proves the result.
\end{proof}

\section{Analysis of the generalization error}\label{sec:gen}

In the previous section, we have established that the solutions of scalar conservation laws with parametric flux and initial data can be efficiently approximated using ReLU neural networks, without suffering from the curse of dimensionality. In practice, one must try to retrieve (or \emph{train}) this neural network based on an available, finite data set. Details of this training procedure can be found in the next subsection. Next, we estimate how accurately this \textit{trained} neural network approximates the true solution. This will be done by proving an upper bound on the so-called cumulative generalization error.

\subsection{Training neural networks}

For $d\in\mathbb{N}$, let $Y=[0,1]^d$ be the space of parameters of both the flux and initial data (i.e. we simplify notation by replacing $Y\times Z$ with $Y$) and let $(Y, \Sigma, \mu)$ be a complete probability space. For $T>0$, the goal is to accurately approximate the mapping $y\in Y \mapsto u(T,\cdot,y)$ by a neural network $u_\theta$ (in the sense of Definition \ref{def:sol-associated-dnn}) where $\theta\in\Theta$ are the parameters of the network. To do so, we draw $M\in\mathbb{N}$ random parameters from $Y$ according to $\mu$, which we will denote by $\S = (y_1, \ldots, y_M)$, and create the \textit{training set} $\mathbb{S} = \{(y,u(T,\cdot,y)) \:\vert\: y\in \S\}$. Using this training set, we define the loss function 
\begin{equation}
    \mathcal{J}(\theta;\S,\lambda) = \frac{1}{M}\sum_{i=1}^M\norm{u(T,\cdot,y_i)-u_{\theta}(\cdot,y_i)}_{L^1(D)} + \lambda \mathcal{R}(\theta),
\end{equation}
where the first sum measures the discrepancy between the true solution and its approximation and the last term is a regularization term that prevents overfitting and hence improves the generalization capabilities of the network. The network hyperparameter $\lambda > 0$ controls the extent of regularization. A popular choice of the regularization function is $\mathcal{R}(\theta) = \norm{\theta}^p_p$ for $p=1,2$. The \textit{training} of the neural network (finding the best neural network approximation) then boils down to finding the parameter $\theta^*(\S)$ that minimizes the loss function: 
\begin{equation}
    \theta^*(\S) = \arg \min_{\theta\in\Theta} \mathcal{J}(\theta;\S,\lambda). 
\end{equation}
This possibly highly non-convex optimization problem is usually solved using an algorithm based on stochastic gradient descent. If the training was successful, the network $u_{\theta^*(\S)}$ will be a good approximation of $u$ on $\S$, but there is no a priori guarantee that $u_{\theta^*(\S)}$ is close to $u$ on $Y\setminus \S$. This is the topic of the next section. 

\subsection{Estimate on the cumulative generalization error}

In this section we will prove that if the training set is large enough, the neural network $u_{\theta^*(\S)}$ will on average be a good approximation of the true solution $u$. In order to do so, we will introduce some quantities. 
For a vector of weights and biases $\theta\in\Theta$, we define the training error 
\begin{equation}\label{eqn:training-error}
    \Et(\theta,\S) = \frac{1}{M}\sum_{i=1}^M\norm{u(T,\cdot,y_i)-u_{\theta}(\cdot,y_i)}_{L^1(D)}
\end{equation}
and the generalization error
\begin{equation}\label{eqn:generalization-error}
    \Eg(\theta) = \int_Y \norm{u(T,\cdot,y)-u_{\theta}(\cdot,y)}_{L^1(D)} d\mu(y). 
\end{equation}
The training error is nothing more than the loss function without the regularization term and quantifies the accuracy of $u_\theta$ on $\S$. We will be particularly interested in the choice $\theta=\theta^*(\S)$, where $\theta^*(\S)$ is the vector that minimizes the loss function. In this case, the generalization error quantifies how well the approximation $u_{\theta^*(\S)}$ generalizes from $\S$ to $Y$. Note that we are interested in the generalization error, but we have only access to the training error. Intuitively, one can generally not expect to infer bounds on the generalization error from the training error as the training set might be ill-conditioned, e.g. it might be that $y_1=y_2=\ldots=y_M$. However, we can hope for finding a connection between $\Et(\theta,\S)$ and $\Eg(\theta)$ that holds with a certain probability or that holds averaged over all training sets. Keeping this in mind, we define the cumulative training error
\begin{equation}\label{eqn:acc-training-error}
    \Aet = \int_{Y^M}\Et(\theta^*(\S),\S) d\mu^M(\S)
\end{equation}
and the cumulative generalization error
\begin{equation}\label{eqn:acc-generalization-error}
    \Aeg = \int_{Y^M}\Eg(\theta^*(\S)) d\mu^M(\S). 
\end{equation}

To prove an upper bound on $\Aeg$ we first decompose the error as such,
\begin{equation}
    \Eg(\theta^*(\S)) \leq \abs{\Eg(\theta^*(\S))-\Et(\theta^*(\S),\S)} + \Et(\theta^*(\S),\S),
\end{equation}
and then take the expectation under $\mu^M$ and the supremum over all $\theta\in\Theta$ to obtain,
\begin{equation}\label{eq:error-decomposition}
    \Aeg \leq \Aet+ \mathbb{E}\left[\sup_{\theta\in\Theta}\abs{\Eg(\theta)-\Et(\theta,\S)}\right].
\end{equation}
This inequality provides an easy recipe to bound the cumulative generalization error. We assume that $\Aet$ can be made arbitrarily small by letting the stochastic gradient descent algorithm run for enough epochs. There are more rigorous upper bounds available \cite{beck2020error,jentzen2020overall}, but these tend to be large overestimates of the true value of $\Aet$. Alternatively, it is common practice to retrain networks multiple times. One could then estimate $\Aet$ by the average of the training errors. Hence, it only remains to find an upper bound on the left term of the right-hand side of \eqref{eqn:bound-Eg}. This term is often called the generalization gap, as it measures the discrepancy between generalization and training error. We will find an upper bound for the generalization gap by adapting \cite[Corollary 4.15]{jentzen2020overall} to our setting. To do so, we introduce the notion of the covering number of a metric space, which will play a crucial role in bounding the generalization gap. In addition, we prove a number of auxiliary results in Appendix \ref{app:gen}. 

\begin{definition}\label{def:covering}
Let $(E,d)$ be a metric space and let $r\in[0,+\infty]$. Then we refer to 
\begin{equation}
    \mathcal{C}^{(E,d)}_r := \inf\left\{n\in\mathbb{N}_0\: \vert\: \exists A \subseteq E : (\abs{A}= n) \land (\forall x \in E: \exists a\in A: d(x,a)\leq r) \right\}
\end{equation}
as the covering number of $(E,d)$.
\end{definition}

\begin{theorem}\label{thm:generalization-gap}
Let $R\geq 1$, $k,d,W,L,M\in\mathbb{N}$, $a,b,\alpha,\beta\in\mathbb{R}$ with $a<b$ and $\beta-\alpha\geq 1$ and let $\Theta$ be the $k$-dimensional parameter space corresponding to $(\alpha,\beta)$-clipped ReLU neural networks $\Phi_\theta$ such that $\mathcal{W}(\Phi_\theta)\leq W$, $\mathcal{L}(\Phi_\theta)\leq L$ and $\mathcal{B}(\Phi_\theta)\leq R$ for all $\theta \in \Theta$ (cf. Definitions \ref{def:relu-nn} and \ref{def:clipped-dnn}). Let $D=[a,b]$ and let $D\times [0,1]^d$ be the domain of the solution maps of the associated neural networks (cf. Definition \ref{def:sol-associated-dnn}). Let $\S$ be the training set drawn according to $\mu^M$ such that $\abs{\S}=M$. For $\theta\in\Theta$, let $\Et(\theta,\S)$ and $\Eg(\theta)$ be given by \eqref{eqn:training-error} and \eqref{eqn:generalization-error}. It then holds that
\begin{align}
    \begin{split}
&\mathbb{E}\left[\sup_{\theta\in\Theta}\abs{\Et(\theta,\S)-\Eg(\theta)}\right]  \leq \frac{12(\beta-\alpha)(b-a)L(W+1)\sqrt{\ln\left(2M^{1/3L}R(W+1)\right)}}{\sqrt{M}}. 
    \end{split}
\end{align}
Moreover, the following simplified bound also holds, 
\begin{equation}
\mathbb{E}\left[\sup_{\theta\in\Theta}\abs{\Et(\theta,\S)-\Eg(\theta)}\right] \leq  \frac{12(\beta-\alpha)(b-a)L(W+1)^2\sqrt{\ln\left(2MR\right)}}{\sqrt{M}}.
\end{equation}
\end{theorem}

\begin{proof}
We first generalize the proof of \cite[Lemma 4.13]{jentzen2020overall}. For every $\theta\in\Theta$, the measurability of the mapping $Y^M\to\mathbb{R}:\S\to\Et(\theta,\S)$ follows from Lemma \ref{lem:measurability}. 
We have that for every $\theta\in\Theta$ it holds that $\Et(\theta,\{y_i\})$, $1\leq i\leq M$, are i.i.d. random variables and therefore
\begin{equation}
    \mathbb{E}\left[\Et(\theta,\S)\right] = \Eg(\theta). 
\end{equation}
Furthermore the assumption that $u_\theta, u\in[\alpha,\beta]$ together with definitions \eqref{eqn:training-error} and \eqref{eqn:generalization-error} ensure that for all $\theta\in\Theta$ and $1\leq i\leq M$ that
\begin{equation}
    \left(\mathbb{E}\left[\sup_{\theta\in\Theta}\abs{\Et(\theta,\{y_i\})-\Eg(\theta)}^p\right]\right)^{1/p} \leq (\beta-\alpha)(b-a). 
\end{equation}
In addition, Lemma \ref{lem:Et-lipschitz} with $p=1$ tells us that $\Et(\theta,\S)$ is Lipschitz in $\theta$ with a Lipschitz constant of at most $(b-a)L^*$ where $L^* := LR^{L-1}(W+1)^L$. 
Combining all the previous observations allows us to use \cite[Corollary 4.12]{jentzen2020overall} to conclude that for any $C>0$, $q\geq 2$,
\begin{align}
\begin{split}
& \left(\mathbb{E}\left[\sup_{\theta\in\Theta}\abs{\Et(\theta,\S)-\Eg(\theta)}^q\right]\right)^{1/q} \\
&\qquad\leq \left(\mathcal{C}^{(\Theta,\norm{\cdot}_\infty)}_{\frac{C(\beta-\alpha)(b-a)\sqrt{q-1}}{(b-a)L^*\sqrt{M}}}\right)^{1/q}\left(\frac{2(C+1)(\beta-\alpha)(b-a)\sqrt{q-1}}{\sqrt{M}}\right),\\
\end{split}
\end{align}
which is the equivalent of \cite[Lemma 4.13]{jentzen2020overall}. We proceed by following the steps from the proof of \cite[Proposition 4.14]{jentzen2020overall}, mutatis mutandis. We restrict ourselves to the main steps and refer to \cite{jentzen2020overall} for the detailed calculations. Lemma \ref{lem:covering} gives us that for $r>0,$
\begin{equation}
    \mathcal{C}^{(\Theta,\norm{\cdot}_\infty)}_r \leq \max\{1,(2R/r)^{k}\}. 
\end{equation}
Following the steps from \cite[Proposition 4.14]{jentzen2020overall} with $\kappa_C \leftarrow 2RL^*\sqrt{M}/(C(\beta-\alpha))$ and $C \leftarrow 1$, we obtain that
\begin{equation}
\mathbb{E}\left[\sup_{\theta\in\Theta}\abs{\Et(\theta,\S)-\Eg(\theta)}\right] \leq \frac{4(\beta-\alpha)(b-a)\sqrt{e\max\{1,k\ln(4R^2M(L^*)^2)\}}}{\sqrt{M}}
\end{equation}
Next, we roughly follow the steps of \cite[Corollary 4.15]{jentzen2020overall}. Note that $k \leq LW(W+1)\leq L(W+1)^2$ and $4L^2\leq 2^2\cdot 2^{2(L-1)}\leq 2^{2L}$ to see that
\begin{align}
\begin{split}
k\ln(4R^2M(L^*)^2) &\leq L(W+1)^2 \ln\left(4R^2M(L^*)^2\right) \\
& =  L(W+1)^2\ln\left(4ML^2R^{2L}(W+1)^{2L}\right)\\
& \leq  L(W+1)^2\ln\left(M2^{2L}R^{2L}(W+1)^{2L}\right)
\end{split}
\end{align}
Since it holds that $L,R,W\geq 1$ we also have that 
\begin{align}
    \begin{split}
        \ln\left(M2^{2L}R^{2L}(W+1)^{2L}\right) &= 3L\ln\left(\left[M2^{2L}R^{2L}(W+1)^{2L}\right]^{1/3L}\right)\\
        &\leq 3L \ln\left(2M^{1/3L}R(W+1)\right).
    \end{split}
\end{align}
Combining the previous observations with the fact that $2M^{1/3L}R(W+1)\geq e$ and therefore $3L \ln\left(2M^{1/3L}R(W+1)\right)\geq 1$ gives us
\begin{align}
    \begin{split}
&\mathbb{E}\left[\sup_{\theta\in\Theta}\abs{\Et(\theta,\S)-\Eg(\theta)}\right] \\
&\qquad \leq \frac{4(\beta-\alpha)(b-a)\sqrt{e\cdot 3L^2(W+1)^2 \ln\left(2M^{1/3L}R(W+1)\right)}}{\sqrt{M}}\\
& \qquad \leq \frac{12(\beta-u)(b-a)L(W+1)\sqrt{\ln\left(2M^{1/3L}R(W+1)\right)}}{\sqrt{M}},
    \end{split}
\end{align}
which gives us our main result. We now further simplify the bound. First notice that 
\begin{equation}
    2(W+1) \leq 2\cdot 2^{(W+1)-1} \leq 2^{W+1},
\end{equation}
leading us to the bound
\begin{align}
    \begin{split}
        \ln\left(2M^{1/3L}R(W+1)\right) & \leq (W+1) \ln\left(\left[2M^{1/3L}R2^{W+1}\right]^{1/(W+1)}\right)\\
        & \leq (W+1)\ln(2MR).
    \end{split}
\end{align}
We finally obtain 
\begin{equation}
\mathbb{E}\left[\sup_{\theta\in\Theta}\abs{\Et(\theta,\S)-\Eg(\theta)}\right] \leq  \frac{12(\beta-\alpha)(b-a)L(W+1)^2\sqrt{\ln\left(2MR\right)}}{\sqrt{M}}.
\end{equation}

\end{proof}

Combining \eqref{eq:error-decomposition} with Theorem \ref{thm:generalization-gap} then provides an elegant, yet explicit and computable upper bound on the cumulative generalization error $\Aeg$, as $\Aet$ can be estimated by e.g. the average of the training errors for multiple random seeds. 

\begin{corollary}\label{cor:total-bound-eg}
Assume the setting and notation of Theorem \ref{thm:generalization-gap}. It holds that,
\begin{equation}\label{eqn:bound-Eg}
    \Aeg \leq \Aet+  \frac{12(\beta-\alpha)(b-a)L(W+1)^2\sqrt{\ln\left(2MR\right)}}{\sqrt{M}}. 
\end{equation}
\end{corollary}

\section{Numerical experiments}\label{sec:numex}

We present some numerical experiments to empirically validate the theoretical results from the previous section. In particular, we will show that the generalization error grows at most polynomially as a function of the dimension of the parameter space for a neural network of fixed size. 

\subsection{Methodology}

We will train networks that approximate the mapping from $(x,y)$ to $u(T,x,y)$ for $T=0.1$. Instead of calculating the $L^1$-norms in the training \eqref{eqn:training-error} and generalization error \eqref{eqn:generalization-error} exactly, we make the approximation 
\begin{equation}
    \norm{u(T,\cdot,y)-u_{\theta}(\cdot,y)}_{L^1(D)} \approx \sum_{j=1}^J \abs{u(T,x_j,y)-u_{\theta}(x_j,y)},
\end{equation}
where $x_j$ are $J=100$ equidistant points on $[a,b]$. Hence, $M$ iid generated parameters $y_1, \ldots, y_M \sim U([0,1]^d)$ give rise to a training set of size $M\cdot J$. In the first experiment, we will consider random initial data and a fixed flux and vice versa in the second experiment. For both experiments, we consider networks with $2,4,6,8$ hidden layers with each $5,10,15,20$ neurons and we choose the network architecture that minimizes the median validation error over 5 runs for $d=4$. In our experiment, this is a DNN with 4 hidden layers and 20 neurons per layer. We then use this fixed architecture for all input dimensions and report the training and generalization error averaged over 5 runs. We minimize the training error \eqref{eqn:training-error} using Adam \cite{kingma2014adam} and train the network for $10^4$ epochs. 

\subsection{Fixed flux}
We first consider the one-dimensional inviscid Burgers' equation, i.e. scalar conservation law \eqref{eqn:original-scl} with $f(u)=u^2/2$. We take the initial data in the form of a truncated Karhunen-Loève expansion,
\begin{equation}
    u_0(x,y) = 1 + \sum_{k=1}^d y_k 2^{1-k} \sin(kx), 
\end{equation}
where $x\in [0,1], y\in[0,1]^d$. In order to create the training set, we generate $M$ independent realizations of $y$ and do the same for the test set. We generate the training (and test) samples using the Rusanov flux, second-order reconstruction with the minmod limiter and SSP2-RK time stepping. An example of $u_0(\cdot, y)$ and $u(T,\cdot, y)$ for some $y$ is given in Figure \ref{fig:example-initial}. In the same figure, the prediction by the trained neural network is shown as well. 
\begin{figure}
    \centering
    \begin{subfigure}{.49\textwidth}
     \includegraphics[width=\textwidth]{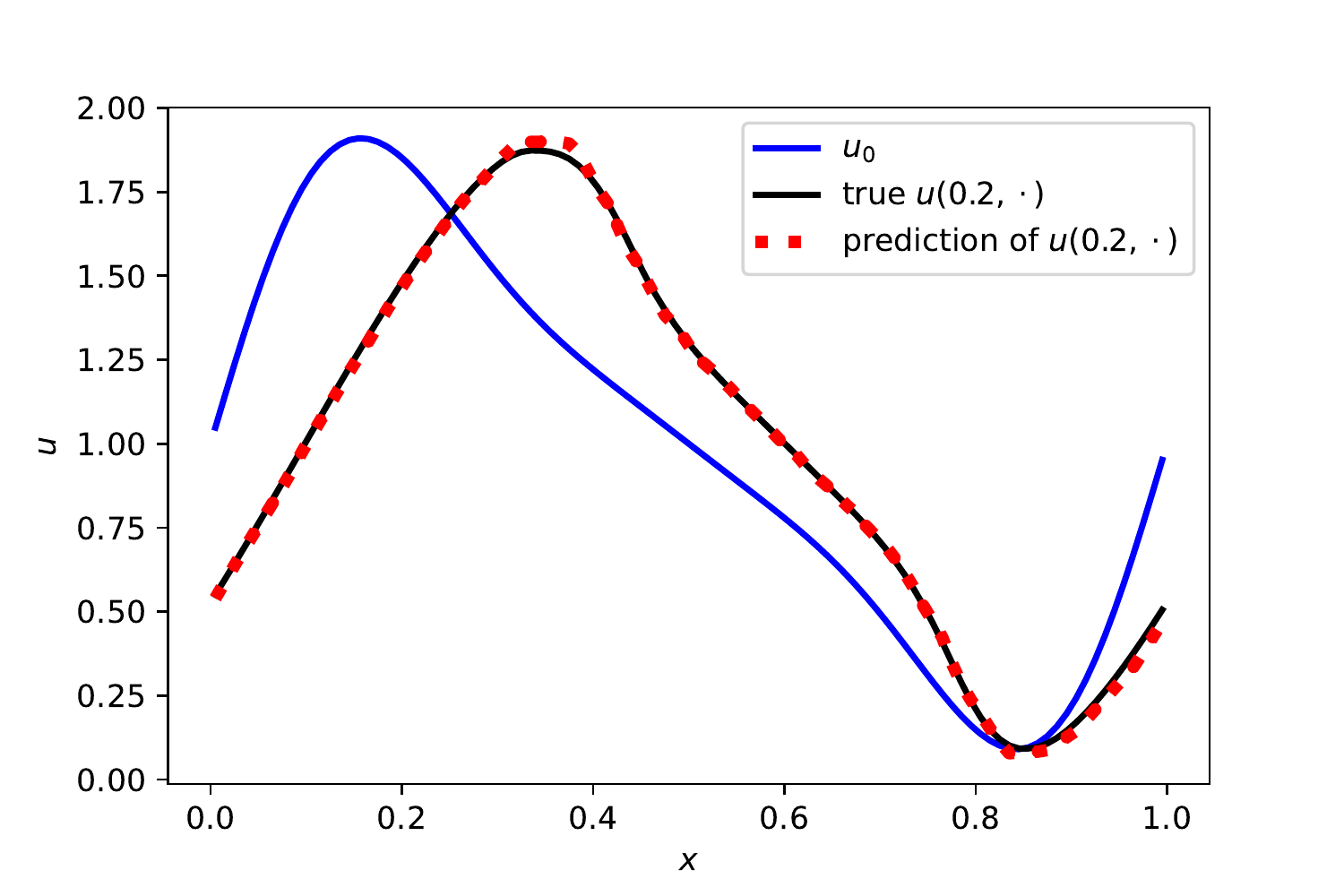}
    \end{subfigure}
  \begin{subfigure}{.49\textwidth}
     \includegraphics[width=\textwidth]{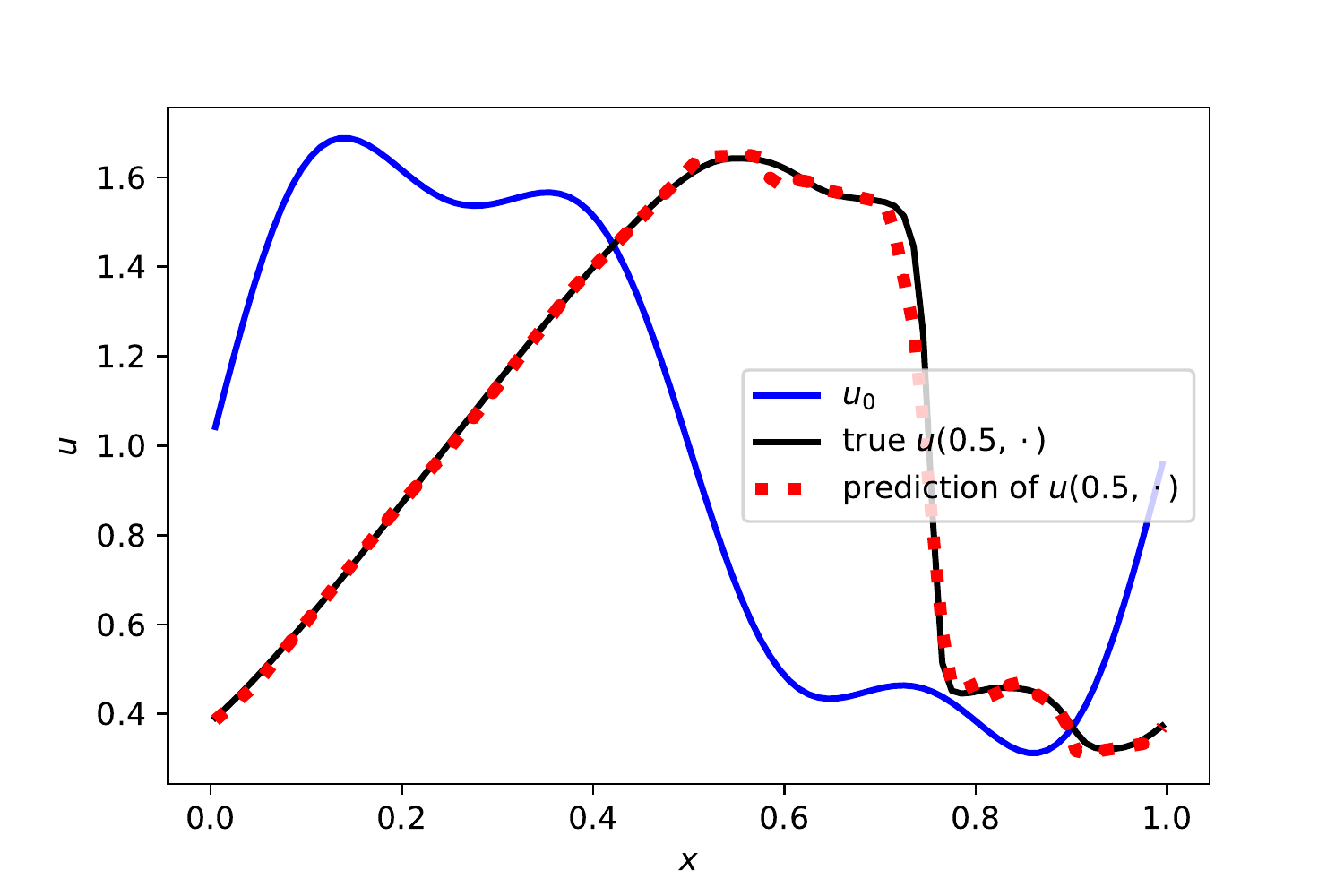}
    \end{subfigure}
    \caption{Example of random initial data, the corresponding true solution at $T=0.2$ and $T=0.5$ and the solution predicted by the neural network.}
    \label{fig:example-initial}
\end{figure}
In Figure \ref{fig:initial-d}, we examine how the training and generalization error depend on the parameter dimension $d$ for two training sets of different size. In both cases, the logarithmic plot reveals that the training and generalization errors only polynomially depend on the parameter dimension $d$. For the smallest training set, of size $M=50$, the generalization error is high for large $d$ and significantly larger than the training data. This signals that the network is not trained in a satisfactory way, in this case due to insufficient training data. For $M=500$, the errors are much smaller and the network generalizes well. 
\begin{figure}
    \centering
    \begin{subfigure}{.49\textwidth}
       \includegraphics[width=\textwidth]{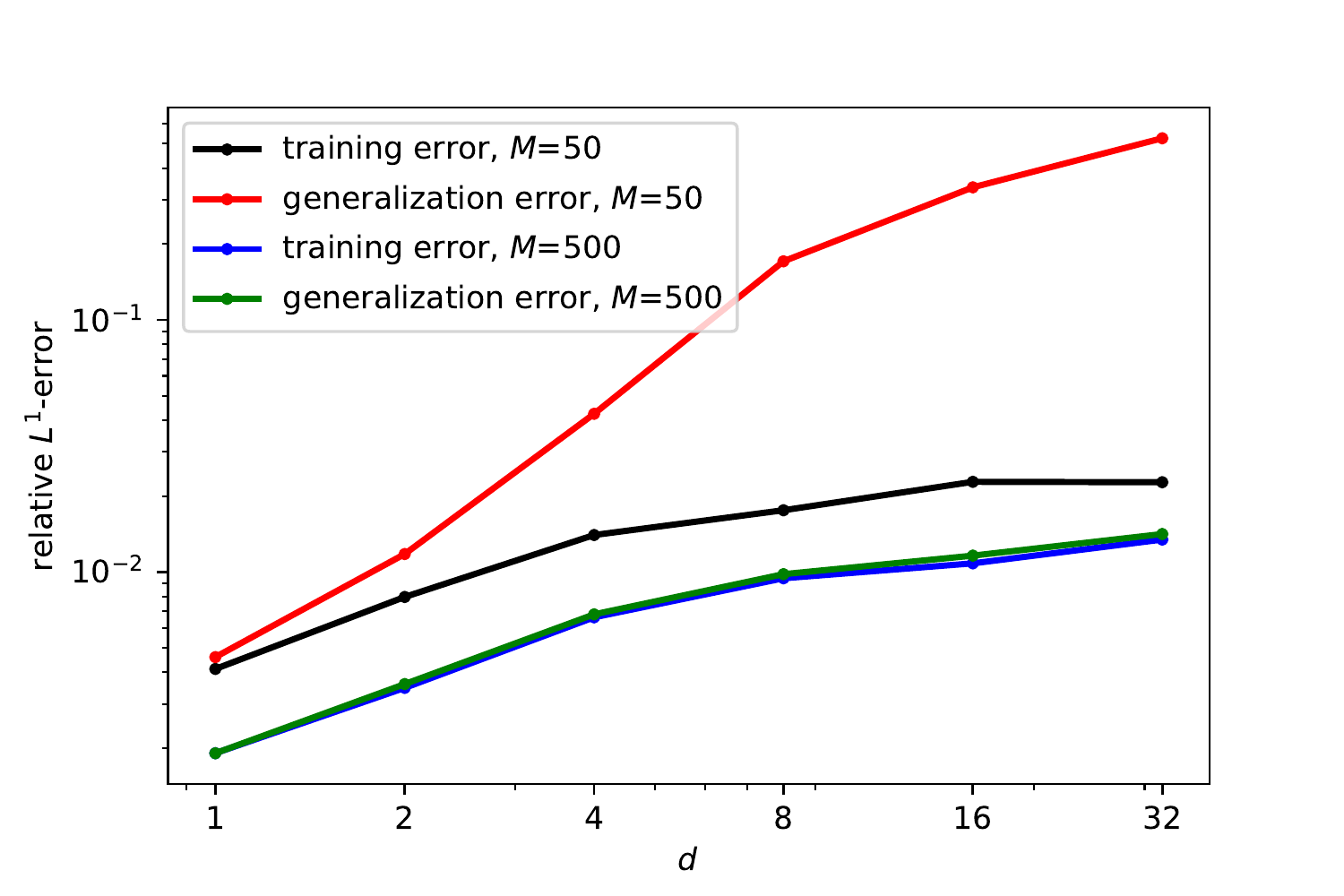}
    \caption{Dependence of the training and generalization error on $d$ for a DNN approximating the solution of Burgers' equation with parameteric initial data. }
    \label{fig:initial-d}
    \end{subfigure}
    \begin{subfigure}{.49\textwidth}
    \includegraphics[width=\textwidth]{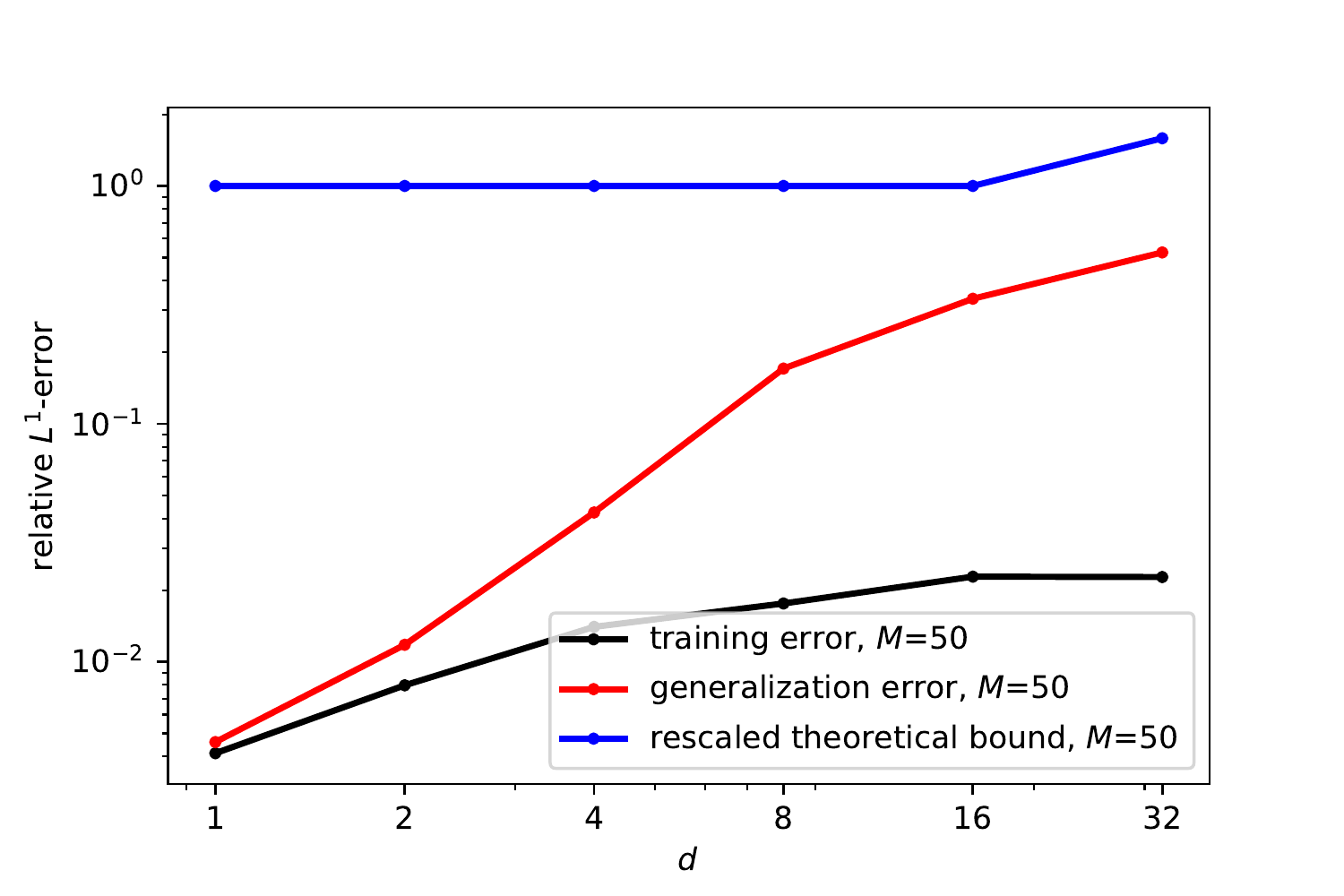}
    \caption{Training and generalization error for $M=50$ in the fixed flux case. A rescaled version of our theoretical bound is shown (rescaled by a factor of $9\cdot 10^2$). }
    \label{fig:comparison-bound2}
    \end{subfigure}
    \caption{Experiments for Burgers' equation with parameteric initial data with fixed flux.}
\end{figure}



Figure \ref{fig:comparison-bound2} also shows a comparison between a theoretical bound and the empirically observed errors. We use Corollary \ref{cor:total-bound-eg} with $L=4$ and $W = \max\{d,20\}$ for the generalization gap and Corollary \ref{cor:scl-dnn-1d-kle} to get an estimate on the approximation error. The theoretical error is rescaled by a factor of $9\cdot 10^2$ for clarity, making clear that the bound is a gross overestimate, as is common for these type of machine learning estimates.

From Figure \ref{fig:initial-d}, it was already clear that the generalization gap decreases as the size of the training set increases. To further explore this observation, we train DNNs for $d=8$ and 
$M=25,50,100,200,400,800$ and repeat this 25 times. In Figure \ref{fig:M-dependence}, the average training and generalization errors (and standard errors) are shown. It is clear that the generalization gap decreases as $M$ grows large. As a consequence, the generalization error converges to the training error. 

\begin{figure}
    \centering
    \includegraphics[scale=0.6]{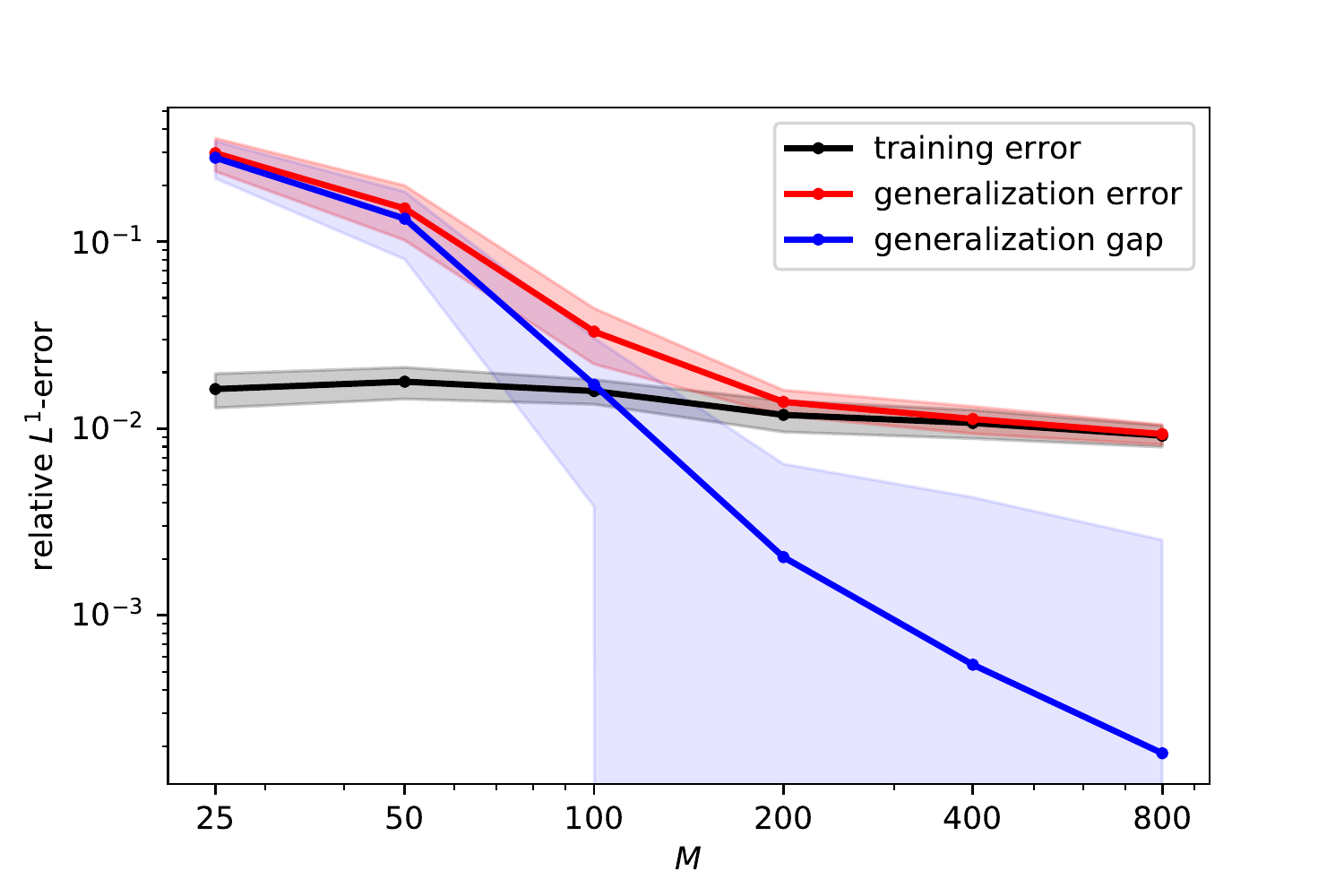}
    \caption{Dependence of the average training error, generalization error and generalization gap on $M$ for $d=8$ and 25 realizations. The shaded area indicates the standard error. }
    \label{fig:M-dependence}
\end{figure}



\subsection{Fixed initial condition}

We repeat the experiment from the previous section for the case where the flux function is random, but the initial data is fixed. We choose for $x\in [0,1]$,
\begin{equation}
    u_0(x) = 1 + \sin(x). 
\end{equation}
Following the numerical experiment in \cite{mishra2016numerical}, the random component of the flux function will be approximated by a Karhunen-Loève expansion, of which the eigenvalues $\lambda_k$ and eigenfunctions $\varphi_k$ are those given by 
\begin{equation}\label{eqn:eqn-kle}
    \int_{A} C(u_1,u_2)\varphi_k(u_1) du_1 = \lambda_k \varphi_k(u_1), \quad u_2 \in [0,2], 
\end{equation}
where $C$ is for instance an exponential covariance kernel, 
\begin{equation}
    C(u_1,u_2) = \sigma^2 \exp{-\abs{u_1-u_2}/\eta},
\end{equation}
where $\sigma = 1$ and $\eta=3$. We then consider flux functions of the form
\begin{equation}
    f(u,z) = \frac{u^2}{2} + \sum_{k=1}^\s z_k \sqrt{\lambda_i} \varphi_k(u),
\end{equation}
where $z\in[0,1]^\s$. In \cite[Section 6.1]{mishra2016numerical} it is demonstrated that the eigenvalues decay as $\lambda_i \sim i^{-2.5}$, which is in agreement with our assumptions on the fast decay of Karhunen-Loève eigenvalues. In addition, the eigenfunctions are linear combinations of sines and cosines. 

A typical realization of this flux function is shown in Figure \ref{fig:flux-example}. In order to create the training set, we generate $M$ independent realizations of $z$ and do the same for the test set. We generate the training (and test) samples using the Godunov flux, second-order reconstruction with the minmod limiter and SSP2-RK time stepping. 

\begin{figure}
    \centering
    \begin{subfigure}{0.49\textwidth}
    \includegraphics[width=\textwidth]{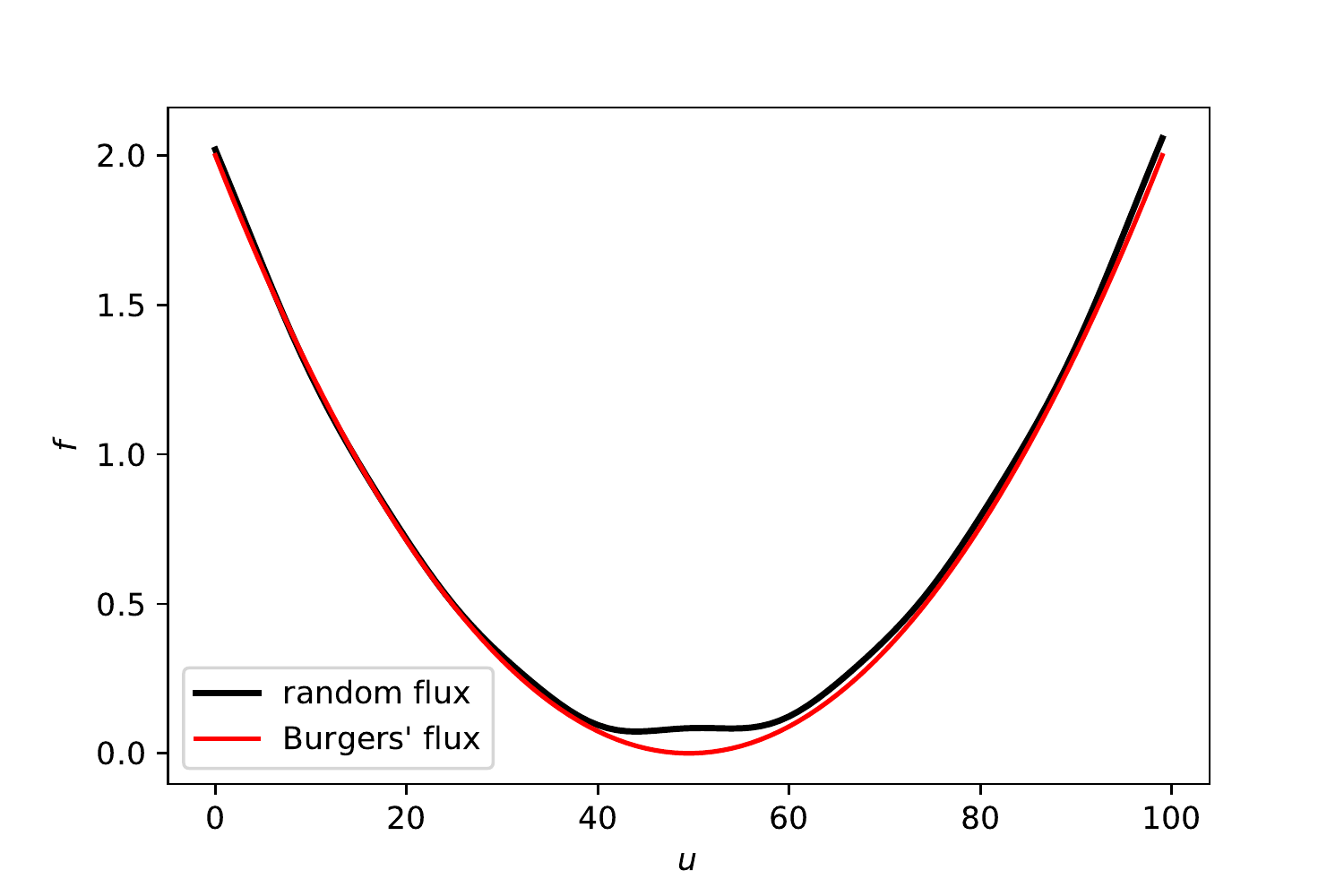}
    \caption{Typical example of a random flux function together with the unperturbed flux function of the inviscid Burgers' equation. }
    \label{fig:flux-example}
    \end{subfigure}
    \begin{subfigure}{0.49\textwidth}
    \includegraphics[width=\textwidth]{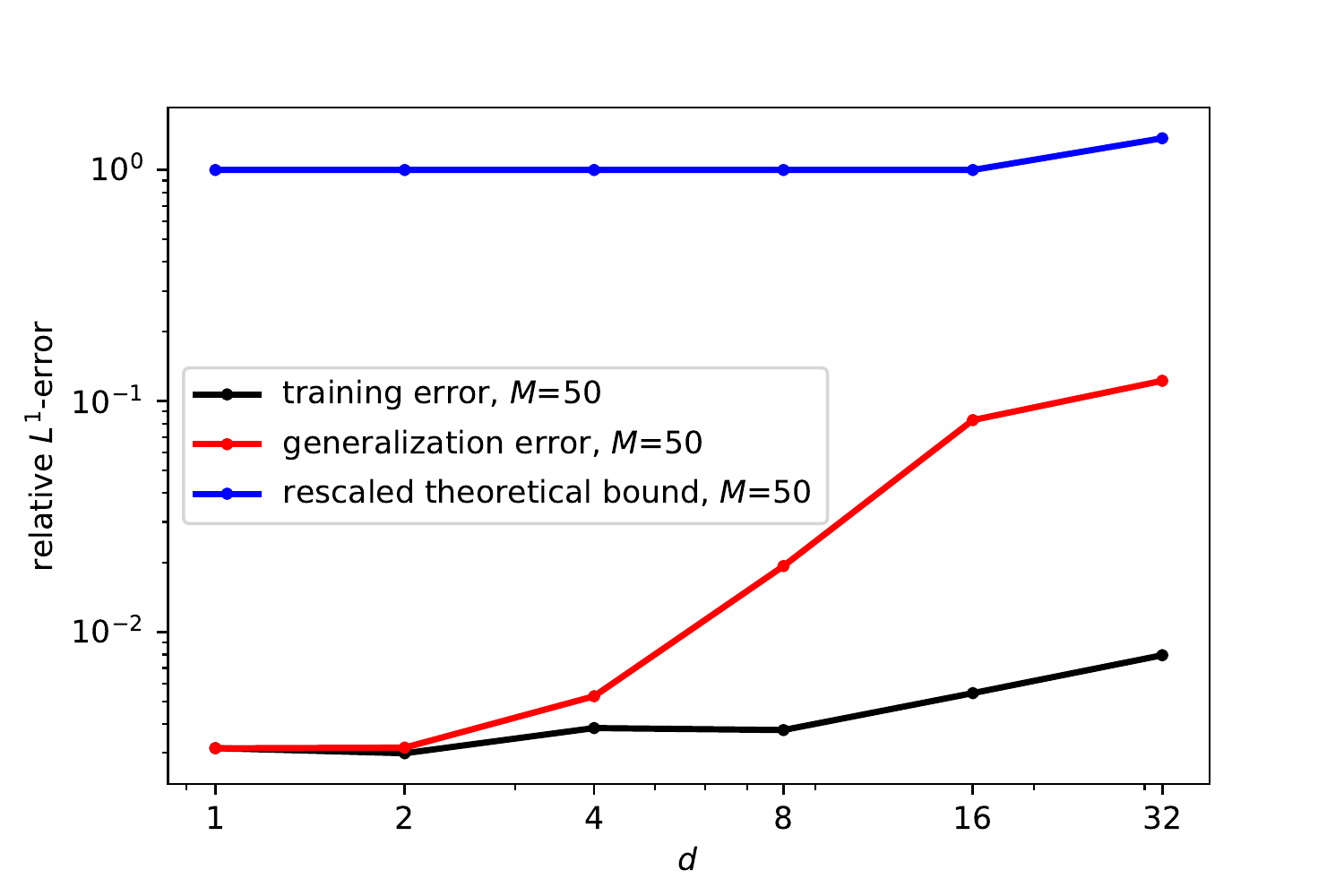}
    \caption{Dependence of the training and generalization error on $d$. Our theoretical bound is shown for comparison, after rescaling with a factor of $1\cdot 10^3$. }
    \label{fig:flux-d2}
    \end{subfigure}
    \caption{Experiments for Burgers' equation with fixed initial data and parametric flux.}
\end{figure}

In Figure \ref{fig:flux-d2}, we see that the dependence of the training and generalization error on $d$ is only polynomially, as in the previous experiment. This time the network already generalizes well for $M=50$. A comparison between our theoretical error bound and the empirically observed errors reveals a large difference that is often seen for such bounds. Note that the theoretical bound in Figure \ref{fig:flux-d2} is rescaled by a factor of $1\cdot 10^3$ for clarity. As in the fixed flux case, the bound massively overestimates the empirically observed error. 




\section{Discussion}

ReLU neural networks are widely used in the approximation of PDE solutions and have been proven to overcome the curse of dimensionality (CoD) in a number of settings. However, there is a paucity of theoretical results on the neural network approximation of (parametric) hyperbolic conservation laws. The contributions of this paper to overcoming this paucity can be summarized as follows,
\begin{itemize}
    \item Under mild assumptions and for parametric initial data that can be approximated by a ReLU neural network without incurring the CoD in the parameter dimension, we have proven that ReLU neural networks overcome the CoD for scalar conservation laws with parametric initial data (Theorem \ref{thm:scl-dnn-1d}). The error bound is explicit and the bounds on the weights and biases are independent of the chosen accuracy. The proof relies amongst others on the convergence of the Lax-Friedrichs method and a stability result for hyperbolic conservation laws. 
    \item Theorem \ref{thm:scl-dnn-flux-1d} proves under similar assumptions that CoD in the parameter dimensions can also be overcome when the flux function is parametric as well. Multiple extensions are considered, such as a neural network approximation in space-time (Corollary \ref{cor:scl-dnn-1d-spacetime}) and an approximation for multi-dimensional scalar conservation laws (Corollary \ref{cor:scl-dnn-multidim}). 
    \item By adapting the results of \cite{jentzen2020overall} to our setting, we prove a rigorous upper bound on the generalization gap i.e., the gap between generalization and training error (Theorem \ref{thm:generalization-gap}). This allows us to prove an estimate on the (accumulative) generalization error in terms of the (accumulative) training error, the number of training samples and the size of the neural network. 
\end{itemize}
The theoretical error bounds are illustrated by a number of numerical experiments in Section \ref{sec:numex}. These clearly demonstrate how the training and generalization error only grow polynomially in terms of the parameter dimension and how the gap between generalization and training error converges to zero, when when the number of training samples is increased, as predicted by the theoretical results. 

We conclude our discussion by pointing out possible limitations and suggestions for further research,
\begin{itemize}
    \item We have only focused on neural netwok approximations in the \emph{supervised learning} paradigm. In this setting, for every parameter in the training set one needs to (approximately) solve the conservation law, e.g. by using a traditional numerical method, which might still be computationally expensive. Fortunately, it is possible to extend to ideas of the paper to suitable unsupervised learning methods such as \emph{weak PINNs} \cite{deryck2022weak}, as well as to operator learning methods using the generic error estimates of \cite{deryck2022generic}. 
    \item Corollary \ref{cor:total-bound-eg} does not prove that the generalization error can be made arbitrarily small, it only provides a computable upper bound in terms of the size of the training set and network, as well as the (cumulative) training error. It is straightforward to argue that the latter is small when the optimization algorithm finds a global minimum and the training set and hypothesis space are large enough \cite[Remark 3.12]{deryck2022navierstokes}. 
    More rigorous results are available under mild assumptions, e.g. \cite{beck2020error,jentzen2020overall}, but are generally worst-case upper bounds that are very large overestimates of the true value. 
    \item The numerical experiments in Section \ref{sec:numex} show that Corollary \ref{cor:total-bound-eg} is an overestimate. One could investigate whether this bound can be made sharper using different techniques. In addition, one could try to lift the restriction to \emph{clipped} ReLU neural networks. Extensions to other bounded activation functions such as tanh are immediate using results from \cite{deryck2021approximation,guhring2021approximation}.
    \item We consider the case of parametric conservation laws, which arise in many different applications such as in UQ and optimal design. However, in many other applications, one will encounter situations where such parametrizations are not necessarily available \cite{deeponets}. In such contexts, it is imperative to learn the entire solution operator, for instance corresponding to the scalar conservation law. This is the domain of \emph{operator learning} and rigorous error bounds for one such operator learning framework (DeepONets) for scalar conservation laws, adapting techniques already considered here, have been presented in \cite{LMK1}. 
\end{itemize}

\bibliographystyle{abbrv}
\bibliography{ref}

\appendix

\section{Additional material for Section \ref{sec:approximation-error}}\label{app:app}

\subsection{Proof of Lemma \ref{lem:flux-approximation}}\label{app:flux-approximation}

\begin{proof}
Let $J\in\mathbb{N}$ a natural number. We will explicitly construct $\Phi^J$ as a linear interpolation of $f$ and prove that the resulting network satisfies the claimed bounds. Define
\begin{equation}
    \mathcal{X}_J = \left\{(x_{-1},\ldots,x_{J+1}) \:\vert\: x_{-1}+\abs{f(a)} = a = x_0 < \ldots < x_J = b = x_{J+1}-\abs{f(b)} \right\}.
\end{equation}
Note that every vector in $\mathcal{X}_J$ defines a (not necessarily uniform) grid on $[a,b]$, complemented by two additional grid points outside of $[a,b]$. 
For every $j\in\{0,\ldots, J\}$ we define the hat function $\rho_j$ through
\begin{equation}
    \rho_j(x) = \frac{\sigma(x-x_{j-1})-\sigma(x-x_j)}{x_j-x_{j-1}} - \frac{\sigma(x-x_{j})-\sigma(x-x_{j+1})}{x_{j+1}-x_{j}}. 
\end{equation}
Note that the functions $\rho_j$ form a partition of unity on $[a,b]$, i.e. $\sum_j\rho_j \equiv 1$ on $[a,b]$. Define 
\begin{align}
    \begin{split}
        \Phi^J(x) & = \sum_{j=0}^J f(x_j)\rho_j(x) \\
        & = \sum_{j=0}^J \left(\frac{f(x_{j+1})-f(x_{j})}{x_{j+1}-x_{j}}-\frac{f(x_{j})-f(x_{j-1})}{x_{j}-x_{j-1}}\right)\sigma(x-x_j) \\
        & \quad + \frac{f(a)}{\abs{f(a)}} \sigma(x-x_{-1}) + \frac{f(b)}{\abs{f(b)}} \sigma(x-x_{J+1}).
    \end{split}
\end{align}
This rewriting reveals that $\Phi^J$ is a ReLU neural network with one hidden layer consisting of $J+3$ neurons and $3J+10$ weights and biases all of which are bounded by $\max \{1,\abs{a}+\abs{f(a)},\abs{b}+\abs{f(b)},2 \norm{f'}_\infty\}$. 

We proceed by bounding the approximation error made when approximating $f$ by $\Phi^J$.
For every $j\in\{0,\ldots, J\}$, let $\xi_j\in[x_{j-1},x_j]$ be the point that satisfies $f'(\xi_j)(x_{j}-x_{j-1}) = f(x_{j})-f(x_{j-1})$, the existence of which is guaranteed by the mean value theorem. Now denote by $(\Phi^J)'$ the weak derivative of $\Phi^J$. One then has that 
\begin{align}
\begin{split}
\abs{f(x)-\Phi^J(x)-(f(y)-\Phi^J(y))} &= \abs{\int_y^x(f'(t)-(\Phi^J)'(t))dt} \\
&\leq \norm{f'-(\Phi^J)'}_\infty \abs{x-y}.   
\end{split}
\end{align}
By Taylor's theorem, there exists for every $x\in [x_{j-1},x_j]$ a $\zeta_x \in[x_{j-1},x_j]$ such that $f'(x) = f'(\xi_j) + f''(\zeta_x)(x-\xi_j)$. Combining this with the fact that $(\Phi^J)'(x) = f'(\xi_j)$ for all $x\in (x_{j-1},x_j)$, we get that
\begin{equation}
    \norm{f'-(\Phi^J)'}_\infty = \max_j \sup_{x\in (x_{j-1},x_j)}\abs{f'(x)-f'(\xi_j)} \leq \max_j \abs{x_j-x_{j-1}} \sup_{x\in (x_{j-1},x_j)} \abs{f''(x)}.
\end{equation}
This estimate holds for any grid in $\mathcal{X}_J$. Combining the previous inequalities, we obtain
\begin{align}
\begin{split}
 \norm{f-\Phi^J}_{\mathrm{Lip}([a,b];\mathbb{R})} & \leq \norm{f'-(\Phi^J)'}_\infty \\&\leq \inf_{(x_{-1},\ldots,x_{J+1})\in\mathcal{X}_J} \max_j \abs{x_j-x_{j-1}} \sup_{x\in (x_{j-1},x_j)} \abs{f''(x)} \\&\leq \frac{b-a}{J}\norm{f''}_\infty, 
\end{split}
\end{align}
where we used that the error made using the optimal grid is bounded by the error on a uniform grid. 
\end{proof}

\subsection{Proof of Lemma \ref{lem:mult-lipschitz}}\label{app:mult-lipschitz}

\begin{proof}
In \cite[Prop. 2 and 3]{yarotsky2017error}, a parametrized class of ReLU neural networks $f_m$, $m\in\mathbb{N}$, that approximate $f:[0,1]\to[0,1]:x\mapsto x^2$ is constructed. In particular, it holds that $\mathcal{M}(f_m) = \mathcal{L}(f_m) = O(m)$ and $\mathcal{W}(f_m) =\mathcal{B}(f_m) = 4$. 
For $m\in\mathbb{N}$, $f_m$ is piecewise affine and for $x\in I^m_k=[k/2^m,(k+1)/2^m]$, $0\leq k\leq 2^m-1$ given by 
\begin{equation}
    f_m(x) =
    \left(\frac{k}{2^m}\right)^2+\left(x-\frac{k}{2^m}\right)\frac{2k+1}{2^m}.
\end{equation}
Similar to the proof of Lemma \ref{lem:flux-approximation} and using the above expression, we obtain that
\begin{align}
    \begin{split}
\norm{f-f_m}_{\mathrm{Lip}([0,1];\mathbb{R})} & \leq \max_k  \norm{f'-f_m'}_{L^\infty(I^m_k;\mathbb{R})} \\ &=  \max_k \sup_{x\in I^m_k}  \abs{2x-\frac{2k+1}{2^m}} = \frac{1}{2^m}.
    \end{split}
\end{align}
In the spirit of \cite[Prop. 3]{yarotsky2017error} and based on the observation that $xy = \frac{1}{4}((x+y)^2-(x-y)^2)$ we then define $\widehat{\times}_m$ for $(x,y)\in[-M,M]\times[-N,N]$ by
\begin{equation}
    \widehat{\times}_m(x,y) = \frac{(M+N)^2}{4}\left(f_m\left(\frac{\abs{x+y}}{M+N}\right)-f_m\left(\frac{\abs{x-y}}{M+N}\right)\right). 
\end{equation}
Since the map $x\mapsto \abs{x\pm y}/(M+N)$ is $1/(M+N)$-Lipschitz, we get for all $y\in[-N,N]$ that
\begin{equation}
    \norm{\times(\cdot,y)-\widehat{\times}_m(\cdot,y)}_{\mathrm{Lip}([-M,M];\mathbb{R})}\leq 2\cdot \frac{(M+N)^2}{4} \cdot \frac{1}{M+N} \cdot \norm{f-f_m}_{\mathrm{Lip}([0,1];\mathbb{R})} = \frac{M+N}{2^{m+1}}.
\end{equation}
Finally, it follows directly that $\mathcal{M}(f_m) = O(m)$, $\mathcal{L}(f_m) = m+1$, $\mathcal{W}(f_m) = 8$ and $\mathcal{B}(f_m) = O((M+N)^2)$.
\end{proof}

\subsection{Convergence rate of Lax-Friedrichs scheme}\label{app:kuznetsov}

We follow Example 3.14 from \cite{holden2015front} to obtain an explicit error bound on the numerical approximation of a scalar conservation law based on the Lax-Friedrichs scheme. This is result is used in Theorem \ref{thm:scl-dnn-1d}.

\begin{lemma}\label{lem:LxF-rate}
Let $u$ be the solution to the scalar conservation law \eqref{eqn:original-scl} with a flux $f\in C^1(\mathbb{R})$ and $u_0\in BV([a,b])$. Let $u_{\Delta t}$ be the numerical approximation to $u$ based on the Lax-Friedrichs scheme with time step $\Delta t$. It holds that
\begin{equation}
    \norm{u_{\Delta t}(\cdot,T)-u(\cdot,T)}_{L^1([a,b])} \leq 31 \mathrm{TV}(u_0) \frac{T(1+\norm{f'}_{\infty})^{2}}{\sqrt{N}}.
\end{equation}
\end{lemma}
\begin{proof}
We follow Example 3.14 from \cite{holden2015front} and apply it specifically to the Lax-Friedrichs scheme. We redo the calculations of \cite[p. 91]{holden2015front} (using their notation) to get an explicit expression for the there mentioned constants. We denote define $F = \norm{f'}_{\infty}$ and denote by $\lambda>0$ the CFL number such that $\Delta t = \lambda \Delta x$. In this case we let $\lambda = F^{-1}$. First note that $p=0$ and $p'=1$ for Lax-Friedrichs. This then gives
\begin{align}
    \begin{split}
        -\Lambda_{\epsilon,\epsilon_0}(u_{\Delta t}, u) &\leq 4 \text{TV}(u_0) F \sum_{n=0}^{N-1}\left(\frac{1}{2}(p+p'+1)^2 \frac{(\Delta x)^2}{\epsilon}+\frac{(\Delta t)^2}{\epsilon_0}+\lambda\left(\frac{(\Delta x)^2}{\epsilon}+\frac{\Delta x\Delta t}{\epsilon_0}\right)\right)\\
        & = 8 \text{TV}(u_0) F N (\Delta t)^2 \left(\frac{1}{\lambda^2\epsilon}+\frac{1}{\epsilon_0}+\lambda\left(\frac{1}{\lambda^2\epsilon}+\frac{1}{\lambda\epsilon_0}\right)\right) \\
        & = 8 \text{TV}(u_0) F T\Delta t \left(\frac{2}{\epsilon_0}+\frac{F^2+F}{\epsilon}\right)\\
        & \leq 16 \text{TV}(u_0) (F+1)^3 \left(\frac{1}{\epsilon_0}+\frac{1}{\epsilon}\right)T\Delta t. 
    \end{split}
\end{align}
Next we want to find an upper bound for 
\begin{equation}
    \nu_t(u_{\Delta t},\epsilon) = \sup_{\abs{\tau}\leq \epsilon} \norm{u_{\Delta t}(\cdot,t+\tau)-u_{\Delta t}(\cdot,t)}_1. 
\end{equation}
We have that for the Lax-Friedrichs scheme
\begin{align}\label{eqn:lxf-lipschitz-time}
    \begin{split}
&\Delta x \sum_j \abs{U^{n+1}_j-U^n_j} \\
&\qquad=  \Delta x \sum_j \abs{\frac{1}{2}(U^n_{j-1}+U^n_{j+1}) - \frac{\Delta t}{2 \Delta x}(\widehat{f}(U^n_{j+1})-\widehat{f}(U^n_{j-1}))-U^n_j}\\
&\qquad\leq \frac{\Delta x}{2} \sum_j \left(\abs{U^n_{j-1}-U^n_j}+\abs{U^n_{j+1}-U^n_j}+\lambda\norm{f}_{\text{Lip}}\abs{U^n_{j+1}-U^n_{j-1}}\right) \\
& \qquad \leq  \frac{\Delta t}{2\lambda} (1 +\lambda F) \sum_j \left(\abs{U^n_{j-1}-U^n_j}+\abs{U^n_{j+1}-U^n_j}\right)\\
& \qquad \leq 2F \text{TV}(u_0)\Delta t.
    \end{split}
\end{align}
As a direct consequence, we obtain
\begin{equation}
    \nu_t(u_{\Delta t},\epsilon) \leq 2F \text{TV}(u_0)(\epsilon+\Delta t).
\end{equation}
Following \cite{holden2015front}, choose the initial approximation such that
\begin{equation}
    \norm{u_{\Delta t}(\cdot,0)-u_{0}(\cdot)}_1 \leq \Delta x \text{TV}(u_0). 
\end{equation}
Now recalling Kuznetsov's lemma \cite[Theorem 3.11]{holden2015front}, 
\begin{align}
\begin{split}
  \norm{u_{\Delta t}(\cdot,T)-u(\cdot,T)}_1 \leq &\norm{u_{\Delta t}(\cdot,0)-u_0(\cdot)}_1 + \text{TV}(u_0)(2\epsilon+\epsilon_0 \norm{f}_{\text{Lip}}) \\ &+ \frac{1}{2}(\nu_T(u_{\Delta t},\epsilon_0)+\nu_0(u_{\Delta t},\epsilon_0))  -\Lambda_{\epsilon,\epsilon_0}, 
\end{split}
\end{align}
and all the previous inequalities, we get that
\begin{align}
\begin{split}
  \norm{u_{\Delta t}(\cdot,T)-u(\cdot,T)}_1 \leq &  \text{TV}(u_0)\big[(\Delta x + 2\epsilon+\epsilon_0 \norm{f}_{\text{Lip}}) \\ &+ 2F(\epsilon_0+\Delta t) + 16(F+1)^3 \left(\frac{1}{\epsilon_0}+\frac{1}{\epsilon}\right)T\Delta t \big]\\
  \leq &\text{TV}(u_0)\big[ 3F \Delta t + 3(1+F)(\epsilon + \epsilon_0) + 16(F+1)^3 \left(\frac{1}{\epsilon_0}+\frac{1}{\epsilon}\right)T\Delta t\big]. 
\end{split}
\end{align}
Minimizing with respect to $\epsilon,\epsilon_0$ then gives us 
\begin{align}
\begin{split}
  \norm{u_{\Delta t}(\cdot,T)-u(\cdot,T)}_1 
  \leq & \text{TV}(u_0)\left[3F \Delta t + 4\cdot 4(1+F)^2\sqrt{3T\Delta t}\right]\\
   = & \text{TV}(u_0)\left[\frac{3FT}{N} + 16T (1+F)^2\sqrt{\frac{3}{N}}\right]\\
   \leq & 31 \text{TV}(u_0) \frac{T(1+F)^2}{\sqrt{N}}. 
\end{split}
\end{align}
\end{proof}

\subsection{Replacing the multiplication operator with a ReLU network of fixed architecture}\label{sec:star-operation}

The cornerstone of many constructive ReLU network approximations is the approximation of the multiplication operator as proposed by Yarotsky \cite{yarotsky2017error}. As per usual, the size of the network corresponding to this approximation depends on the desired approximation accuracy. We discuss an alternative method to approximate the calculation of a convex combination using a ReLU neural network of which the size does not depend on the approximation accuracy, as proposed in \cite{deryck2020approximation} based on Yarotsky's work. This method is used in Corollary \ref{cor:scl-dnn-1d-spacetime}.

\begin{definition}\label{def:star}
For $\lambda>0$, we denote by $\star$ the operation given by 
\begin{equation}
    \star: [0,1]\times[-\lambda,\lambda]\to [-\lambda,\lambda]: (x,y)\mapsto x \star y := \sigma(y+\lambda x-\lambda)-\sigma(-y+\lambda x-\lambda), 
\end{equation}
where $\sigma$ is the ReLU activation function (cf. Definition \ref{def:relu-act}). 
\end{definition}
We compare $x\star y$ with $x\cdot y$ for fixed $x\in[0,1]$ and $\lambda>0$ in Figure \ref{fig:star}. In addition, the following properties hold. 
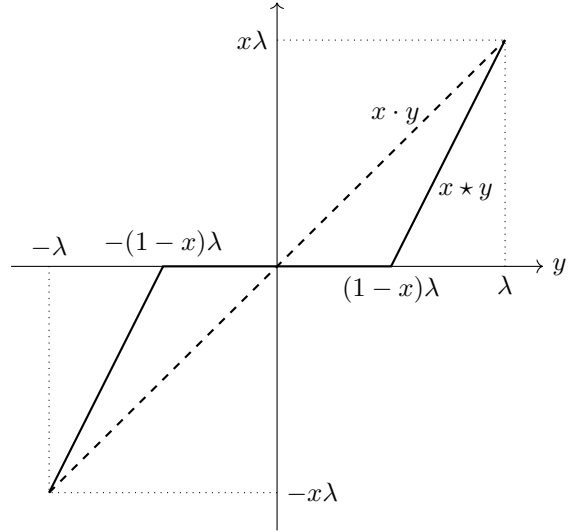
\begin{figure}
    \centering
    \begin{tikzpicture}

\draw[->] (-3.5,0) -- (3.5,0) node[right] {$y$};
\draw[->] (0,-3.5) -- (0,3.5) node[right] {};
\draw[dashed,thick] (-3,-3)--(3,3);
\draw[thick] (-3,-3)--(-1.5,0)--(1.5,0)--(3,3);
\node[left] at (2,2) {$x\cdot y$};
\node[right] at (2,1) {$x\star y$};
\draw[dotted] (0,3)--(3,3)--(3,0);
\draw[dotted] (0,-3)--(-3,-3)--(-3,0);
\node[below] at (1.5,0) {$(1-x)\lambda$};
\node[below] at (3,0) {$\lambda$};
\node[above] at (-1.5,0) {$-(1-x)\lambda$};
\node[above] at (-3,0) {$-\lambda$};
\node[left] at (0,3) {$x\lambda$};
\node[right] at (0,-3) {$-x\lambda$};
\end{tikzpicture}
    \caption{Plot of $x\star y$ and $x\cdot y$ for fixed $x\in[0,1]$ and $\lambda>0$.}
    \label{fig:star}
\end{figure}

\begin{lemma}\label{lem:star}
For $\lambda>0$, the operation $\star$ satisfies the following properties: 
\begin{enumerate}
    \item For all $x\in\{0,1\}$ and $y\in[-\lambda,\lambda]$ it holds true that $x \star y = xy$.
    \item For all $x\in[0,1]$ and $y\in[0,\lambda]$ we have $0\leq x \star y \leq xy$.
    \item For all $x\in[0,1]$ and $y\in[-\lambda,0]$ we have $xy\leq x \star y \leq 0$.
    \item There exist $x\in[0,1]$ and $y_1,y_2\in[-\lambda,\lambda]$ such that \[\min\{y_1,y_2\}\leq (1-x)\star y_1 + x \star y_2\leq \max\{y_1,y_2\}\] does not hold. 
    \item For all $x\in[0,1]$ and $y_1,y_2\in[-\lambda,\lambda]$ it holds true that \[\min\{y_1,y_2\}\leq y_1 + x \star (y_2-y_1)\leq \max\{y_1,y_2\}.\]
\end{enumerate}
\end{lemma}

In conclusion, suppose we need to approximate $xy_1+(1-x)y_2$ where $x\in[0,1]$ and $y_1,y_2\in[-\lambda,\lambda]$ for some $\lambda>0$. 
Lemma \ref{lem:star} then assures us that $y_1 + x \star (y_2-y_1) = xy_1+(1-x)y_2$ for $x\in\{0,1\}$. Furthermore, if $0<x<1$ then $y_1 + x \star (y_2-y_1)$ is still a convex combination of $y_1$ and $y_2$. For applications where this suffices, the operation $\star$ provides a more efficient alternative to Yarotsky's approximation of the multiplication operator. 

\section{Additional material for Section \ref{sec:gen}}\label{app:gen}

We prove some auxiliary results on the measurability of the training error and the Lipschitz continuity of the training error with respect to the parameters of the ReLU neural network. 

\begin{lemma}\label{lem:measurability}
Let $T>0$, $\theta\in\Theta$ and let $D\times Y\to \mathbb{R}:(x,y) \mapsto u(T,x,y)$ and $D\times Y\to \mathbb{R}:(x,y) \mapsto u_\theta(x,y)$ measurable mappings.  The the mapping $Y\to \mathbb{R}: y\mapsto \norm{u(T,\cdot,y)-u_{\theta}(\cdot,y)}_{L^1(D)}$ is $\Sigma / \mathcal{B}([0,\infty]$-measurable. 
\end{lemma}
\begin{proof}
Let $T>0$ and $\theta\in\Theta$. The result follows directly from the measurability of $u(T,\cdot,\cdot)$ and $u_\theta$, the Fubini-Tonelli theorem and the fact that the composition of continuous and measurable functions is again measurable. 
\end{proof}

\begin{lemma}\label{lem:1NN-lipschitz}
Let $d,L,W,R\in\mathbb{N}$, $a,b,\alpha,\beta\in\mathbb{R}$ with $a<b$ and $\alpha+1\leq \beta$, $D=[a,b]$, $Y=[0,1]^d$ and let $\Theta$ be the parameters of the class of $(\alpha,\beta)$-clipped ReLU neural networks $\Phi_\theta:D\to[\alpha,\beta]$ (cf. Definitions \ref{def:relu-nn} and \ref{def:clipped-dnn}), with $\mathcal{L}(\Phi_\theta)=L$, $\mathcal{W}(\Phi_\theta)=W$ and such that $\mathcal{B}(\Phi_\theta)\leq R$. Then it holds for all $\theta, \eta\in\Theta$ that 
\begin{equation}
    \norm{\Phi_\theta-\Phi_\eta}_{L^{\infty}} \leq L R^{L-1}(W+1)^L\norm{\theta-\eta}_{\infty}. 
\end{equation}
\end{lemma}
\begin{proof}
This is \cite[Cor. 2.37]{beck2020error}. 
\end{proof}
In the following, we use the notation from Lemma \ref{lem:1NN-lipschitz} without explicitly defining it again. 
\begin{lemma}\label{lem:ourNN-lipschitz}
For $y\in Y$, it holds that
\begin{equation}
    \norm{u_{\theta_1}(\cdot,y)-u_{\theta_2}(\cdot,y)}_{L^1(D)} \leq L(b-a)R^{L-1}(W+1)^L\norm{\theta_1-\theta_2}_{\infty}. 
\end{equation}
\end{lemma}
\begin{proof}
Because we identified the output of neural networks with functions (cf. Definition \ref{def:sol-associated-dnn}), the Lipschitz continuity of the associated solution maps with respect to the parameters does not follow immediately from Lemma \ref{lem:1NN-lipschitz}. Recall from Definition \ref{def:sol-associated-dnn}, that the outputs of the actual networks are $u_\theta(x_j,y)$. Hence, Lemma \ref{lem:1NN-lipschitz} can be applied on the subnetworks $y\mapsto u_\theta(x_j,y)$. We then get that
\begin{align}
\begin{split}
    \norm{u_{\theta_1}(\cdot,y)-u_{\theta_2}(\cdot,y)}_{L^1(D)} & = \frac{b-a}{J} \sum_{j=1}^J \abs{u_{\theta_1}(x_j,y)-u_{\theta_2}(x_j,y)} \\
    & \leq (b-a) \cdot L R^{L-1}(W+1)^L\norm{\theta_1-\theta_2}_{\infty}
\end{split}    
\end{align}
\end{proof}

\begin{lemma}\label{lem:Et-lipschitz}
Let $p\in\mathbb{N}$. Define $\mathcal{Y}_{\theta,i} = \Et(\theta,\{y_i\})$. If $u(T,x,y)\in [\alpha,\beta]$ and $\sup_{\theta\in\Theta}u_\theta(x,y)\in [\alpha,\beta]$, then it holds for $\theta_1, \theta_2$ that 
\begin{equation}
    \abs{(\mathcal{Y}_{\theta_1,i})^p-(\mathcal{Y}_{\theta_2,i})^p} \leq p(\beta-\alpha)^{p-1}(b-a)^p LR^{L-1}(W+1)^L \norm{\theta_1-\theta_2}_{\infty}.
\end{equation}
\end{lemma}

\begin{proof}
Let $p\in\mathbb{N}$. For $x,y\in\mathbb{R}$ it holds that $x^p-y^p = (x-y)\sum_{k=1}^px^{k-1}y^{p-k}$. From the assumptions of the lemma statement, it follows that $\abs{\mathcal{Y}_{\theta,i}}\leq (\beta-\alpha)(b-a)$. Therefore, 
\begin{equation}
    \abs{(\mathcal{Y}_{\theta_1,i})^p-(\mathcal{Y}_{\theta_2,i})^p} \leq p (\beta-\alpha)^{p-1}(b-a)^{p-1} \abs{\mathcal{Y}_{\theta_1,i}-\mathcal{Y}_{\theta_2,i}}.
\end{equation}
Using the reverse triangle inequality we obtain that
\begin{align}
    \begin{split}
\abs{\mathcal{Y}_{\theta_1,i}-\mathcal{Y}_{\theta_2,i}} &= \abs{\norm{u(T,\cdot,y_i)-u_{\theta_1}(\cdot,y_i)}_{L^1(D)}-\norm{u(T,\cdot,y_i)-u_{\theta_2}(\cdot,y_i)}_{L^1(D)}}\\
& \leq \int_D \abs{\abs{u(T,x,y_i)-u_{\theta_1}(x,y_i)}-\abs{u(T,x,y_i)-u_{\theta_2}(x,y_i)}} dx \\
& \leq  \int_D\abs{u_{\theta_1}(x,y_i)-u_{\theta_2}(x,y_i)}dx\\
& = \norm{u_{\theta_1}(\cdot,y_i)-u_{\theta_2}(\cdot,y_i)}_{L^1(D)}. 
    \end{split}
\end{align}
Now using Lemma \ref{lem:ourNN-lipschitz}, we find that
\begin{equation}
    \norm{u_{\theta_1}(\cdot,y_i)-u_{\theta_2}(\cdot,y_i)}_{L^1(D)} \leq L(b-a)R^{L-1}(W+1)^L \norm{\theta_1-\theta_2}_{\infty}.
\end{equation}
Combining all previous inequalities gives the wanted result.
\end{proof}

\begin{lemma}\label{lem:covering}
Let $a,b\in\mathbb{R}$, $r>0$ and $k\in\mathbb{N}$. Then it holds that
\begin{equation}
    \mathcal{C}^{([a,b]^{k},\norm{\cdot}_\infty)}_r \leq \left\lceil \frac{b-a}{r}\right\rceil^{k}.
\end{equation}
\end{lemma}

\end{document}